\documentclass[sn-mathphys,Numbered]{sn-jnl}
\usepackage[T1]{fontenc}
\usepackage[utf8]{inputenc}
\usepackage[english]{babel}
\usepackage{latexsym}
\usepackage{amsmath}
\usepackage{multirow}%
\usepackage{amssymb}
\usepackage{amsfonts}
\usepackage{graphicx}
\usepackage{float}
\usepackage{upgreek}
\usepackage{ mathrsfs }
\usepackage{ dsfont }
\usepackage{amsthm}
\usepackage{graphicx}
\usepackage[title]{appendix}%
\usepackage{xcolor}%
\usepackage{booktabs}
\usepackage{manyfoot}%
\usepackage{color}
\usepackage{framed}
\usepackage{verbatim}
\usepackage{pdflscape,rotating,arydshln}
\usepackage{textcomp}
\usepackage{afterpage}
\usepackage{capt-of}
\usepackage{slashbox}
\usepackage{algorithm}%
\usepackage{algorithmicx}%
\usepackage{algpseudocode}%
\usepackage{listings}%
\usepackage{bm}
\newtheorem{theorem}{Theorem}
\newtheorem{lemma}[theorem]{Lemma}
\newtheorem{corollary}[theorem]{Corollary}
\newtheorem{example}{Example}%
\newtheorem{remark}{Remark}%

\newtheorem{definition}{Definition}%

\newcommand{\N}{\mathbb{N}}

\newcommand{\R}{\mathbb{R}}
\newcommand{\C}{\mathbb{C}}

\begin{document}

\title[Fox-$H$ densities and completely monotone generalized Wright functions]{Fox-$H$ densities and completely monotone generalized Wright functions}
 
\author[1]{\fnm{Luisa} \sur{Beghin}}\email{luisa.beghin@uniroma1.it}

\author*[2]{\fnm{Lorenzo} \sur{Cristofaro}}\email{lorenzo.cristofaro@uni.lu}

\author[3]{\fnm{Jos{é} Lu{í}s} \sur{da Silva}}\email{joses@staff.uma.pt}

\affil[1]{\orgdiv{Department of Statistical Sciences}, \orgname{Sapienza}, \orgaddress{\street{Piaz.le Aldo Moro, 5}, \city{Rome}, \postcode{00185}, \country{Italy}}}

\affil*[2]{\orgdiv{Department of Mathematics}, \orgname{University of Luxembourg}, \orgaddress{\street{6, avenue de la Fonte}, \city{Esch-sur-Alzette}, \postcode{4364}, \country{Luxembourg}}}

\affil[3]{\orgdiv{Faculty of Exact Sciences and Engineering}, \orgname{University of Madeira}, \orgaddress{\street{Campus da Penteada}, \city{Funchal}, \postcode{9020-105}, \country{Portungal}}}

\abstract{Due to their flexibility, Fox-$H$ functions are widely studied and applied to many research topics, such as astrophysics, mechanical statistic, probability, etc. Well-known special cases of Fox-$H$ functions, such as Mittag-Leffler and Wright functions, find a wide application in the theory of stochastic processes, anomalous diffusions and non-Gaussian analysis.
In this paper, we focus on certain explicit assumptions that allow us to use the Fox-$H$ functions as densities. We then provide a subfamily of the latter, called Fox-$H$ densities with all moments finite, and give their Laplace transforms as entire generalized Wright functions. The class of random variables with these densities is proven to possess a monoid structure. We present eight subclasses of special cases of such densities (together with their Laplace transforms) that are particularly relevant in applications, thanks to their probabilistic interpretation. To analyze the existence conditions of Fox-$H$ functions red as well as their sign, we derive asymptotic results and their analytic extension.}

\pacs[MSC Classification]{26A48, 33E12, 60E05, 60E10}

\maketitle

 \section{Introduction}

The Fox-$H$ functions were defined in \cite{Fox} for the first time, as well as the subclass of the so-called Meijer's $G$-functions. These functions were exploited to collect a wide set of solutions of special differential equations from physics, whose definitions were given by series representations or Mellin-Barnes integrals; see \cite{ErdI}. The analytical properties of the Mellin-Barnes integrals were also studied: for example, \cite{DF36} investigated the convergence of a Mellin-Barnes integral expressed by Fox-$H$ functions. A fundamental contribution to this topic can be found in \cite{Bra64}. \\
In the XX century, many authors studied Fox-$H$ functions thanks to their great flexibility: we can find a wide range of applications in statistics, astrophysics, diffusion problems, probability, and operator theory; see \cite{MS78}, \cite{Spri80}, \cite{Kir}, and \cite{Mai22}, among others.  
During the 1980s, some special cases of $G$-functions were shown to be densities in \cite{SprinThomp} and \cite{CarSprin}. More recently, a class of densities expressed using Fox-$H$ functions was introduced in \cite{Gau18}. 
In \cite{VellKat} and \cite{GupKum22} Fox-$H$ functions were studied in connection with L\'{e}vy stable subordinators and their inverses, respectively. In \cite{KatVell18} Fox-$H$ functions were generalized to a larger class, called $I$-functions, to represent some densities.

Furthermore, we also recall that some statistics applications are presented in \cite{Gau18bis}, \cite{Coel19}, \cite{Spri80}.

In the last thirty years, the link to Fractional Calculus has been established and developed: in this framework, a milestone is represented by \cite{Kir}. Many authors studied the interplay between some transcendental functions of Wright type and time-fractional diffusion processes; see \cite{GLM99}.
The relationships among Fox-$H$ functions and Fractional Calculus were also presented in \cite{Karp2018}, \cite{Karp2018Gfunc}; in particular, Mittag-Leffler functions, Wright functions, Le Roy functions and many others are strictly connected to fractional operators, as shown, for example, in~\cite{GKMR20} and \cite{GHG19}.

One of the most interesting properties of these special functions is complete monotonicity, which allows applications to physics and probability; see, for example, \cite{GKMR20}. In particular, a straightforward application of this property can be found in non-Gaussian analysis, see \cite{S90}, \cite{JahnI} and \cite{BCG23}. 
In \cite{Mehr18} and \cite{KarpPril16}, conditions for the complete monotonicity of certain generalized Wright functions were derived.

 In this paper, we provide some new insight into Fox-$H$ functions and generalized Wright functions to identify them as densities or completely monotone functions. The paper is organized as follows. In Section~\ref{sec:DefinitionExistenceAnaliticalContinuation}, we prove new results on the asymptotic behavior of the integrand (in its Mellin-Barnes integral representation) and the analytic continuation of the corresponding $H$-function.
In Section~\ref{sec:FoxfunctionsAsDensities}, we specify the assumptions on the parameters, to identify certain classes of Fox-$H$ densities. Furthermore, we describe two decomposition procedures that enable the study of the sign of a Fox-$H$ function. In particular, we define an algebraic structure for random variables whose distribution is a Fox-$H$ density with all moments finite. Furthermore, we study the asymptotic behaviour at infinity and at zero for the latter. In Section~\ref{sec:FoxHDensities}, we present eight classes of Fox-$H$ densities together with their Laplace
transforms (LT, hereinafter) which can be expressed by generalized Wright functions; in Section~\ref{sec:Examples}, we collect a few examples of well-known special and elementary functions which can be covered by the classes previously defined; in Section~\ref{Sec:Outlook}, we conclude the paper giving some perspectives for future work.

\section{Fox-$H$ functions: definitions and preliminary results}
\label{sec:DefinitionExistenceAnaliticalContinuation}

We recall here the definitions of Fox-$H$ functions (also named in the literature $H$-functions). Moreover, we give statements of some new results concerning the asymptotic behavior of the function $|\mathcal{H}^{m,n}_{p,q}(s)z^{-s}|$ (for $s \to \infty$) involved in the definition of Fox-$H$ functions, as well as the analytical extension of the associated Fox-$H$ function (for given parameters and contour) in Lemma~\ref{lem:asympHCircle} and Theorem~\ref{thm:ResidueDeltaPosi}.
Their proofs can be found in the Appendix~\ref{sec:AppendixAnalyticalresults}. These results will be used in Section~\ref{sec:FoxfunctionsAsDensities}, where we will use the LT of Fox-$H$ functions defined through the contour $\mathcal{L}=\mathcal{L}_{\mathrm{i}\gamma \infty}$; see Theorem~2.3 in \cite{SaiKil}. We refer the interested reader to the monographs \cite{Mathai2009}, \cite{SaiKil}, and \cite{Fox} for more details of this class of functions.

\begin{definition}[cf.~p.1 in \cite{SaiKil}]\label{def:H-function}
For integers $m,n,p,q$ such that $0 \leq m \leq q$, $0 \leq n \leq p$, for $a_i, b_j \in \C$ and for $\alpha_i,\beta_j \in (0,\infty)$ with $i=1,\dots,p $ and $j=1,\dots,q$, the Fox-$H$ function $H_{p,q}^{m,n}(\cdot)$ is defined via a Mellin-Barnes integral as
	
	\[ 
	H_{p,q}^{m,n}(z):=H_{p,q}^{m,n}\left[z \, \Bigg| \genfrac{}{}{0pt}{}{(a_i,\alpha_i)_p}{(b_j,\beta_j)_q}\right]:=\frac{1}{2\pi \mathrm{i}}\int_{\mathcal{L}} \mathcal{H}^{m,n}_{p,q}(s)z^{-s}\,\mathrm{d}s, \quad z \in \mathbb{C}\backslash \{0\},
	\]
	where 
	\[  
	\mathcal{H}^{m,n}_{p,q}(s):= \frac{\prod_{j=1}^{m} \Gamma({b_j+s\beta_j}) \prod_{i=1}^{n}\Gamma({1-a_i-s\alpha_i})}{ \prod_{i=n+1}^{p}\Gamma({a_i+s\alpha_i})\prod_{j=m+1}^{q} \Gamma({1-b_j
			-s\beta_j})}
	\]
	and $\mathcal{L}$ is the infinite contour that separates all the poles  given by Equation~\eqref{def:polesb} from those given by Equation~\eqref{def:polesa}, see below for details.

\end{definition}
The function $\mathcal{H}^{m,n}_{p,q}(\cdot)$ inherits the poles from the functions $\Gamma(b_j+s\beta_j)$
\begin{equation} \label{def:polesb} b_{jl}=-\frac{b_j+l}{\beta_j}\quad j=1,\dots,m, \quad l=0,1,2,\dots \end{equation}
and the functions $\Gamma(1-a_i-s\alpha_i)$
\begin{equation} \label{def:polesa} a_{ik}=\frac{1-a_i+k}{\alpha_i}\quad i=1,\dots,n,\quad k=0,1,2,\dots. \end{equation}

The above poles do not coincide if $\alpha_i(b_j+l)\neq \beta_j(a_i-k-1)$, for $i=1,\dots,n$,  $j=1,\dots,m$, with $k,l \in \N_0$; see Equation~(1.1.6) in \cite{SaiKil}. \\
The contour $\mathcal{L}$ can take three distinct shapes, which we refer to as $\mathcal{L}_{-\infty}$, $\mathcal{L}_{+\infty}$, and $\mathcal{L}_{\mathrm{i}\gamma\infty}$, defined as follows:
\begin{itemize}
	\item $\mathcal{L}=\mathcal{L}_{-\infty}$ is a left loop situated in a horizontal strip starting at the point $-\infty + \mathrm{i} \phi_1$ and terminating at the point $-\infty + \mathrm{i} \phi_2$ with $-\infty<\phi_1<\phi_2<+\infty$;
	\item $\mathcal{L}=\mathcal{L}_{+\infty}$ is a right loop situated in a horizontal strip starting at the point $+\infty + \mathrm{i} \phi_1$ and terminating at the point $+\infty + \mathrm{i}\phi_2$ with $-\infty<\phi_1<\phi_2<+\infty$;
	\item $\mathcal{L}=\mathcal{L}_{\mathrm{i}\gamma\infty}$ is a contour starting at the point $\gamma -\mathrm{i}\infty$ and terminating at the point $\gamma +\mathrm{i}\infty$, where $\gamma \in \mathbb{R}=(-\infty,+\infty)$.
\end{itemize}
Here $z^{-s}=\exp[-s(\mathrm{Log}(|z|) + \mathrm{i}\arg(z))]$, $z \neq0$, where $\mathrm{Log}(|z|)$ represents the natural logarithm of $|z|$ and $\arg(z)$ is the principal value. 

The properties of the Fox-$H$ function depend on the contour $\mathcal{L}$ and the quantities $a^{*}, \Delta, \delta, \mu, a_1^{*}$ and $a^*_2$, which are defined as
$a^*=a_1^*+a_2^*$, $\Delta=a_1^*-a_2^*$, $\delta=\prod_{j=1}^q\beta_j^{\beta_j} \prod_{i=1}^p\alpha_i^{-\alpha_i}$ and $\mu=\sum_{j=1}^q b_j-\sum_{i=1}^p a_i +(p-q)/2$, where $a_1^*=\sum_{j=1}^m\beta_j-\sum_{i=n+1}^p\alpha_i$  and $a_2^*=\sum_{i=1}^n\alpha_i-\sum_{j=m+1}^q\beta_j$; see Equations~(1.1.9), (1.1.10), (1.1.11), (1.1.12) and (1.1.13) in \cite{SaiKil}.

We will employ the following alternative and more explicit notation when necessary:
\[  H_{p,q}^{m,n}\left[z \, \Bigg| \genfrac{}{}{0pt}{}{(a_i,\alpha_i)_{1,p}}{(b_j,\beta_j)_{1,q}}\right]=H_{p,q}^{m,n}\left[z \, \Bigg| \genfrac{}{}{0pt}{}{(a_p,\alpha_p)}{(b_q,\beta_q)}\right]=H_{p,q}^{m,n}\left[z \,\Bigg| \genfrac{}{}{0pt}{}{(a_1,\alpha_1),\dots, (a_p,\alpha_p)}{(b_1,\beta_1),\dots,(b_q,\beta_q)}\right] .\]
In the case $p=1$ or $q=1$ we will use the short notation: $(a_i,\alpha_i)_{1,p}=(a,\alpha)$ or $(b_j,\beta_j)_{1,q}=(b,\beta)$, respectively.  A well-known subclass of Fox-$H$ functions is the Meijer $G$-functions, defined as:
	
	\[ G^{m,n}_{p,q} \left[ x\, \Bigg|\genfrac{}{}{0pt}{}{(a_i)_{1,p}}{(b_j)_{1,q}} \right] :=H^{m,n}_{p,q} \left[ x\, \Bigg|\genfrac{}{}{0pt}{}{(a_i,1)_{1,p}}{(b_j,1)_{1,q}} \right]  .\]

Any empty sum is taken to be zero and any empty product is taken to be $1$.

We recall the definition of the generalized Wright functions, i.e.~$_m\Psi_p(\cdot)$, (a subclass of Fox-$H$ functions also known as Fox-Wright functions) that represents the largest class of special functions that could be represented as the LT of densities using Lemma~\ref{lem:HFiniteDensities} below.

For $m,p \in \N_0$, $A_i,B_j\in \mathbb{C}$ and $\alpha_i,\beta_j \in \R \backslash \{0\}$, with $i=1,\dots,m$ and $j=1,\dots,p$, the generalized Wright  functions are defined by the series \[ _m\Psi_p\left[ \genfrac{}{}{0pt}{}{(A_i, \alpha_i )_{1,m}}{(B_i, \beta_j )_{1,p} }\Bigg| z \right]=\sum_{k\geq0} \frac{\prod_{i=1}^m \Gamma(A_i+\alpha_i k)}{\prod_{j=1}^p \Gamma(B_j+\beta_j k)} \frac{z^k}{k!}, \quad z \in \mathbb{C}.\]
 They could be represented by Fox-$H$ functions if $\alpha_i,\beta_j$ are positive: 
	
	\[  _m\Psi_p\left[ \genfrac{}{}{0pt}{}{(A_i, \alpha_i )_{1,m}}{(B_i, \beta_j )_{1,p} }\Bigg| -z \right]=H^{1,m}_{m,p+1}\left[z \, \Bigg| \genfrac{}{}{0pt}{}{(1-A_i, \alpha_i )_{1,m}}{(0,1),(1-B_j, \beta_j )_{1,p} }\right], \quad z \in \mathbb{C} ,\]
 see Equation (5.2) in \cite{KST02}.

In Section \ref{sec:FoxfunctionsAsDensities} we use the LT of Fox-$H$ functions with $\mathcal{L}=\mathcal{L}_{\mathrm{i}\gamma \infty}$ as a contour to represent generalized Wright functions as Fox-$H$. By Theorem~1.2-$(iii)$ in \cite{SaiKil}, we know that LTs are analytic, but we also need to specify the sign of $\Delta$ to know its expression in series. Therefore, we will recover the series representation of Theorem~1.2-$(iii)$ in \cite{SaiKil} by showing that the case $(1.2.14)$ of Theorem~1.1 in \cite{SaiKil} may be considered analytical continuation of case $(1.2.20)$, for a certain value of $\Delta$ and parameters $a_i$, $b_j$, $\alpha_i$, and $\beta_j$. Theorem~\ref{thm:ResidueDeltaPosi} provides the assumptions for the analytical continuation, and Lemma~\ref{lem:asympHCircle} below is needed for its proof. These results are new, to the best of our knowledge.

We present new results concerning the asymptotic behavior of $|\mathcal{H}^{m,n}_{p,q}(s)z^{-s}|$, for $z \in \C \backslash \{0\}$ and $s=\gamma + R\mathrm{e}^{\mathrm{i}\theta}$, with $R \to \infty$, which is used to prove Theorem~\ref{thm:ResidueDeltaPosi} in Appendix~\ref{sec:AppendixAnalyticalresults}.
\begin{lemma} \label{lem:asympHCircle}
	Let $m,n,p,q \in \mathbb{N}_0$ be such that $0 \leq m \leq q$, $0 \leq n \leq p$, $a_i, b_j \in \mathbb{R}$ and $\alpha_i,\beta_j \in (0,\infty)$, with $i=1,\dots,p $ and $j=1,\dots,q$. If the poles $b_{jl}$ given in \eqref{def:polesb}, $j=1,\dots,m$ and $l \in \mathbb{N}_0$, are on the left of $\mathcal{L}_{\mathrm{i}\gamma \infty}$ and the poles $a_{ik}$ given in \eqref{def:polesa}, $i=1,\dots,n$ and $k \in \mathbb{N}_0$, are on the right of $\mathcal{L}_{\mathrm{i}\gamma \infty}$, with $\gamma \in \mathbb{R}$, then, for $s=\gamma + R\mathrm{e}^{\mathrm{i}\theta}$, with $\theta \in (-\pi,\pi)\backslash\{0\}$, $R>0$ and $z \in \mathbb{C}\backslash\{0\}$, the following asymptotic formula 
	holds:
	
	\begin{equation}\label{eq:HKernelonCircle}
			|\mathcal{H}^{m,n}_{p,q}(s)z^{-s}| \sim  C
			\exp(\Upsilon_1R\mathrm{Log}(R)-\Upsilon_2R + \Upsilon_3\mathrm{Log}(R)),
	\end{equation}
    as $R \to \infty$, where $\Upsilon_1=\cos(\theta)\Delta$, $\Upsilon_2=\cos(\theta)\mathrm{Log}(\exp(\Delta)|z|/\delta)+\sin(\theta)\Theta$ and $\Upsilon_3=\mu+\gamma\Delta$ with $\Theta=a_1^* \theta - a_2^* \mathrm{sign}(\theta)(|\theta|-\pi)+\arg(z)$ and the constant $C$ is given by
	\[C=|z|^{-\gamma}(2\pi)^{m+n-(p+q)/2}\frac{\prod_{j=1}^q \beta_j^{b_j+\beta_j\gamma-1/2} }{\prod_{i=1}^p \alpha_i^{a_i+\alpha_i\gamma-1/2}},\]
	
	and $\Delta, \delta, \mu, a_1^*$ and $a_2^*$ are defined in Equations~(1.1.8), (1.1.9), (1.1.10), (1.1.11) and (1.1.12), respectively, in \cite{SaiKil}.
\end{lemma}
\begin{proof}
    The proof is provided in Appendix~\ref{sec:AppendixAnalyticalresults}.
\end{proof}
 The following theorem specifies that, under certain assumptions on the parameters, the change in contour corresponds to an analytical extension of a Fox-$H$ function $H^{m,n}_{p,q}$.
\begin{theorem}\label{thm:ResidueDeltaPosi}
Let $m,n,p,q \in \mathbb{N}_0$, $a_i, b_j \in \mathbb{R}$ and let $\alpha_i,\beta_j \in (0,\infty)$, with $i=1,\dots,p $ and $j=1,\dots,q$, be such that $\Delta>0$ and $a^*>0$. The Fox-$H$ function $H^{m,n}_{p,q}(\cdot)$ defined on $\mathbb{C}\backslash \{0\}$ by choosing $\mathcal{L}=\mathcal{L}_{-\infty}$
extends analytically $H^{m,n}_{p,q}(z)$ defined on $|\arg(z)|<a^* \pi/2,z\neq0$ by choosing $\mathcal{L}=\mathcal{L}_{\mathrm{i}\gamma \infty}$.  
\end{theorem}

\begin{proof}
    The proof is provided in Appendix~\ref{sec:AppendixAnalyticalresults}.
\end{proof}

 Depending on the the multiplicity of the poles, the residue of $\mathcal{H}^{m,n}_{p,q}(s)z^{-s}$ may lead to different series representation, see Sec.~1.4 in \cite{SaiKil}. As in the following sections we are interested in Fox-$H$ functions of the form $H^{m,0}_{p,m}(x)$ we give its asymptotic behaviour at infinity and at zero in Corollary ~\ref{prop:AsymptotHDensity}.

\section{Fox-$H$ densities and their Laplace transform}
  \label{sec:FoxfunctionsAsDensities}
		In this section, we first introduce the classes of Fox-$H$ densities which will be used in what follows; see Definition \ref{defHdengeneral}. In Lemma~\ref{lem:AllHdensity} we give assumptions on the parameters under which the corresponding Fox-$H$ function becomes a density and also compute its LT. After Lemma~\ref{lem:AllHdensity} we present two procedures for checking the non-negativity of a Fox-$H$ function in Theorem~\ref{thm:FoxHPosi}.
        In Lemma~\ref{lem:HFiniteDensities} we find the class of Fox-$H$ densities with all the moments finite, briefly called FHdam, and we also express the LT of such densities as generalized Wright functions. Then we study the set of random variables in FHdam as an algebraic structure in Corollary~\ref{lem:ClosureLemma} and the asymptotic behavior of FHdam in Corollary~\ref{prop:AsymptotHDensity}. We conclude this section with Examples~\ref{exa:GammaIncompleteDensity} and \ref{exa:threeparameterWrightdensity} where we give a practical application of Lemmas~\ref{lem:AllHdensity} and \ref{lem:HFiniteDensities}, respectively. The LT of a function $f:\R^+_0 \to\R$ defined by $(\mathscr{L}f)(s):=\int_0^\infty \mathrm{e}^{-st}f(t)\,\mathrm{d}t$, $s\geq0$.

\begin{definition}[Def.~3.1~\cite{CarSprin}]
	\label{defHdengeneral} A random variable $X$, on the probability space $(\Omega, \mathcal{F}, \mathbb{P})$, with density function given by
	\begin{equation} \varrho(x)=
		\begin{cases}
			\frac{1}{K}H^{m,n}_{p,q}\left[ c x \, \bigg|\displaystyle{ \genfrac{}{}{0pt}{}{(a_{1},\alpha_{1}),\dots,(a_{p},\alpha_{p})}{(b_{1},\beta_{1}),\dots,(b_{q},\beta_{q})}}\right] ,& x>0,\\
			0,& \text{otherwise,}
		\end{cases}
	\end{equation}
	is called Fox-$H$ random variable and $\varrho(\cdot)$ is its Fox-$H$ density.\\
	The parameters $K,c,a_i,\alpha_i,b_j$ and $\beta_j$, for $i=1,\dots,p$ and $j=1,\dots,q$, are chosen such that $\varrho(x)\geq 0$, $x \in \R$, and
	\[ \int_{0}^{\infty} \varrho(x)\,\mathrm{d}x=1.  \]
	
\end{definition}
\begin{remark}
	In \cite{CarSprin}, the authors use $k=1/K$ and call the above random variable $X$ "Fox-$H$ function variate" instead; see also \cite{MatSax69}. A further insight about the parameters $K, b_j,a_i,\beta_j,\alpha_i$, for $j=1,\dots,m$ and $i=1,\dots,n$, is given in Theorem 4.6 in \cite{Cook}.  
\end{remark}

Some properties of the Fox-$H$ random variables are presented in Appendix~\ref{sec:AppendixFoxHRandomVariable}. The following two lemmas state the assumptions on a Fox-$H$ function and its parameters under which it is a density. After the first lemma, we discuss an approach to verify the sign of a given Fox-$H$ function based on integral decomposition. In the second (Lemma~\ref{lem:HFiniteDensities}) we define the Fox-$H$ density with all the moments finite and present the result on the complete monotonicity of generalized Wright functions.

\begin{lemma}\label{lem:AllHdensity}
			Let $m,n,p,q \in \mathbb{N}_0$ be such that $p\geq n\geq 0$ and $q\geq m \geq 0$. We make the following assumptions on the parameters of the Fox-$H$ function:
			\begin{enumerate}
				\item Let  $a_i \in \mathbb{R}$ and $\alpha_i>0$, for $i=1,\dots,p$, such that $a_i + \alpha_i<1$, for $i=1,\dots,n$, and $a_i+\alpha_i>0$, for $i=n+1,\dots,p$ and that the poles $a_{ik}$ in Equation~\eqref{def:polesa} are simple; \\
				\item Let $b_j \in \mathbb{R}$ and $\beta_j>0$, for $j=1,\dots,q$, such that $b_j+\beta_j>0$, for $j=1,\dots,m$, and $b_j + \beta_j<1$, for $j=m+1,\dots,q$;\\
				\item Let either $a^*>0$ or $a^*=0$, $\mu<-1$;
				\item Let $H^{m,n}_{p,q} \left[ \cdot \,\big|\genfrac{}{}{0pt}{}{(a_i,\alpha_i)_{1,p}}{(b_j,\beta_j)_{1,q}} \right]$ be non-negative on $(0,\infty)$.
			\end{enumerate}
			Then we have that the following function is a Fox-$H$ density on $(0,\infty)$:
			
			\[  \varrho(x)=\frac{1}{K}H^{m,n}_{p,q} \left[ x \,\Bigg|\genfrac{}{}{0pt}{}{(a_i,\alpha_i)_{1,p}}{(b_j,\beta_j)_{1,q}} \right], \quad x>0,\]
			where $K=\mathcal{H}^{m,n}_{p,q}(1)$.\\
			Furthermore, the LT $\phi(\cdot):=(\mathscr{L}\varrho)(\cdot)$ is given, for every $s \geq 0$, by
			\[ \phi(s)=\int_0^\infty \mathrm{e}^{-sx } \varrho(x)\mathrm{d}x=\frac{1}{K}H^{n+1,m}_{q,p+1}\left[ s \,\bigg| \genfrac{}{}{0pt}{}{(1-b_j-\beta_j,\beta_j)_{1,q}}{(0,1),(1-a_i-\alpha_i,\alpha_i)_{1,p}}\right].  \]
		\end{lemma}

		\begin{proof}

	By Assumption 4 about the non-negativity of $H^{m,n}_{p,q}(\cdot)$, we have that 		
			\begin{eqnarray*}   \| H^{m,n}_{p,q}\|_{L^1(\mathbb{R}^+)}&=&\lim_{s \to 0^+} \mathscr{L}(H^{m,n}_{p,q})(s)\\
				&=&\lim_{s \to 0^+} H^{n+1,m}_{q,p+1}\left[ s \, \bigg| \genfrac{}{}{0pt}{}{(1-b_j-\beta_j,\beta_j)_{1,q}}{(0,1),(1-a_i-\alpha_i,\alpha_i)_{1,p}}\right],
			\end{eqnarray*}
			where the second equality follows from Theorem~2.3 in \cite{SaiKil} and Equations~(2.1.3) and (2.1.5) in \cite{SaiKil}.
			 By Assumption 3, we may apply Theorem~1.11 in \cite{SaiKil} to the  Fox-$H$ function on the right-hand side and obtain
			
			\begin{multline}
				\label{seriesexpansion} H^{n+1,m}_{q,p+1}\left[s \, \Bigg| \genfrac{}{}{0pt}{}{(1-b_j-\beta_j,\beta_j)_{1,q}}{(0,1),(1-a_i-\alpha_i,\alpha_i)_{1,p}}\right] \\
				= h_1^*s^0+\sum_{j=2}^{n+1}\big[ h_j^* s^{(1-a_{j-1}-\alpha_{j-1})/\alpha_{j-1}}+o(s^{(1-a_{j-1}-\alpha_{j-1})/\alpha_{j-1}})\big],
			\end{multline}

			where $h_j^*$, $j=2,\dots,n+1$, are given by Equation~(1.8.5) in \cite{SaiKil} and $h_1^*$ by
			\begin{equation}\label{constantspaceweight}  h_1^*=\frac{\prod_{i=1}^{n} \Gamma(1-a_i-\alpha_i) \prod_{i=1}^{m}\Gamma(b_i+\beta_i)}{ \prod_{i=m+1}^{q}\Gamma(1-b_i-\beta_i)\prod_{i=n+1}^{p} \Gamma(a_i + \alpha_i)}=\mathcal{H}^{m,n}_{p,q}(1). \end{equation}
			Using Assumptions 1 and 2 on the parameters, the claim $\| H^{m,n}_{p,q}\|_{L^1(\mathbb{R}^+)}=\mathcal{H}^{m,n}_{p,q}(1)$ follows taking the limit as $s \to 0^+$ in \eqref{seriesexpansion}. This completes the proof of the lemma.		\end{proof}

The assumption $(4)$ in the previous lemma is generally not easy to verify for a given Fox-$H$ function, see \cite{KSW20} and \cite{Kle90}. In \cite{K2021}, Kemppainen investigated the positivity of specific Fox-$H$ functions that arise as Green functions of fractional PDEs. The method in \cite{K2021} is useful for cases where the parameters involved are well identified. Indeed, the argument presented is based on the series representation of the Fox-$H$ function, its Laplace transform, and Bernstein's theorem. However, it seems that there is no direct formula that could be used to check the sign of a Fox-$H$ function.\\
In the following, instead, we propose a decomposition method aimed at deducing the non-negativity of a Fox-$H$ function $H^{m,n}_{p,q} \left[ x \,\big|\genfrac{}{}{0pt}{}{(a_i,\alpha_i)_{1,p}}{(b_j,\beta_j)_{1,q}} \right]$, $x>0$, with $m,n,a^*_{H^{m,n}_{p,q}}>0,$ by exploiting Fox-$H$ functions with a lower number of parameters. Specifically, we provide two methods for decomposing a Fox-$H$ function based on the possible rewriting $(a_i,\alpha_i)_{1,p}$. Analogous approaches could be formulated by rewriting the parameters $(b_j,\beta_j)_{1,q}$.\\
The first procedure, called $\mathbf{A}$, is based on Corollary~2.9.1 in \cite{SaiKil} and is divided into three steps. The first step enables us to identify two Fox-$H$ functions: $H^{m,0}_{0,q}(\cdot)$ and $H^{0,N}_{P,0}(\cdot)$, which represent a given Fox-$H$ function $H^{m,n}_{p,q}(\cdot)$ as an integral transform. The second step (resp.~third step) associates other two possible Fox-$H$ functions, i.e.~$H^{m',0}_{0,q'}(\cdot)$ and $H^{h,0}_{0,h}(\cdot)$ (resp.~$H^{0,N'}_{P',0}(\cdot)$ and $H^{0,g}_{g,0}(\cdot)$), to $H^{m,0}_{0,q}(\cdot)$ (resp.~$H^{0,N}_{P,0}(\cdot)$), see steps $\mathbf{A}\text{-}(2)$ and $(3)$ for details. \\
The second procedure, called $\mathbf{B}$, applies Theorem~2.9 in \cite{SaiKil} to a given Fox-$H$ function $H^{m,n}_{p,q}(\cdot),$ decomposing it through $H^{m,0}_{0,q}(\cdot)$ and $H^{M,0}_{0,Q}(\cdot)$ as an integral transform in the first step. The second step of the procedure $\mathbf{B}$ uses step $\mathbf{A}\text{-}(2)$ on $H^{m,0}_{0,q}(\cdot)$ and $H^{M,0}_{0,Q}(\cdot)$.
\\
In any case, studying the sign of the Fox-$H$ functions $H^{\cdot,0}_{0,\cdot}(\cdot)$ and $H^{0,\cdot}_{\cdot,0}(\cdot)$ is easier.
\begin{description}
\item[$\mathbf{A}$] 
\begin{enumerate}
  \item We represent a given $H^{m,n}_{p,q}$ with $n=N$ and $p=P$ by using Corollary~2.9.1 in \cite{SaiKil}, identifying $(a_i,\alpha_i)_{1,p}=(c^A_i+\eta^A_1 \gamma^A_i, \sigma^A_1 \gamma^A_i)_{1,P}$, such that assumption~(2.8.3) in \cite{SaiKil} holds, as follows
\begin{eqnarray*}
&&H^{m,N}_{P,q} \left[ x \,\Bigg|\genfrac{}{}{0pt}{}{(c_i^A+\eta^A_1 \gamma_i^A,\sigma^A_1 \gamma_i^A)_{1,P}}{(b_j,\beta_j)_{1,q}} \right]\\
&&=\int_0^\infty t^{\eta^A_1 -1} H^{m,0}_{0,q} \left[ x t^{-\sigma^A_1} \,\Bigg|\genfrac{}{}{0pt}{}{{-\!-}\!{-\!-}}{(b_j,\beta_j)_{1,q}} \right]H^{0,N}_{P,0} \left[ t \,\Bigg|\genfrac{}{}{0pt}{}{(c_i^A, \gamma_i^A)_{1,P}}{{-\!-}\!{-\!-}\!{-\!-}} \right]\mathrm{d}t,
\end{eqnarray*}
provided that $a^*_{H^{m,N}_{P,q}}=a^*_{H^{m,0}_{0,q}} + \sigma^A_1 a^*_{H^{0,N}_{P,0}}$ with $\sigma^A_1,a^*_{H^{m,0}_{0,q}}, a^*_{H^{0,N}_{P,0}} >0$.
\item If $a^*_{H^{m,0}_{0,q}}>0$ and $q=lm+h$ with $l \in \N_0$ and $h\in \{1,\dots,m-1\}$ with some $(b_j,\beta_j)_{1,h}=(d_j^A+\eta^A_2\delta_j^A,\sigma^A_2\delta_j^A)_{1,h}$ satisfying assumption~(2.8.3) in \cite{SaiKil}, we apply again Corollary~2.9.1 in \cite{SaiKil} to obtain 
\begin{eqnarray*}
&&H^{m,0}_{0,q} \left[ y \,\Bigg|\genfrac{}{}{0pt}{}{{-\!-}\!{-\!-}\!{-\!-}\!{-\!-}}{(d_j^A+\eta^A_2\delta_j^A,\sigma^A_2\delta_j^A)_{1,h},(b_j,\beta_j)_{h+1,q}} \right]\\
&&=\int_0^\infty t^{\eta^A_2 -1} H^{m',0}_{0,q'} \left[ y t^{-\sigma^A_2} \,\Bigg|\genfrac{}{}{0pt}{}{{-\!-}\!{-\!-}}{(b_{h+j},\beta_{h+j})_{1,q'}} \right]H^{h,0}_{0,h} \left[ t \,\Bigg|\genfrac{}{}{0pt}{}{{-\!-}\!{-\!-}\!{-\!-}}{(d_j^A,\delta_j^A)_{1,h}} \right]\mathrm{d}t
\end{eqnarray*}
 such that $a^*_{H^{m,0}_{0,q}}=a^*_{H^{m',0}_{0,q'}} + \sigma^A_2 a^*_{H^{h,0}_{0,h}}$ with $\sigma^A_2,a^*_{H^{m',0}_{0,q'}}, a^*_{H^{h,0}_{0,h}} >0$, where $m':=m-h$ and $q':=q-h$.\\
\item If $a^*_{H^{0,N}_{P,0}}>0$ and $P=kN+g$ with $k\in\N_0$ and $g\in\{1,\dots,N-1\}$ with some $(c_i^A,\gamma_i^A)_{1,g}=(v^A_i+\eta^A_3 \nu^A_i,\sigma^A_3 \nu^A_i)_{1,g}$ satisfying assumption~(2.8.3) in \cite{SaiKil}, we apply again Corollary 2.9.1 in \cite{SaiKil} to obtain 
\begin{eqnarray*}
&&H^{0,N}_{P,0} \left[ x \,\Bigg|\genfrac{}{}{0pt}{}{(v^A_i+\eta^A_3 \nu^A_i,\sigma^A_3 \nu^A_i)_{1,g},(c_i^A, \gamma_i^A)_{g+1,P}}{{-\!-}\!{-\!-}\!{-\!-}\!{-\!-}} \right]\\
&&=\int_0^\infty t^{\eta^A_3 -1} H^{0,N'}_{P',0} \left[ x t^{-\sigma^A_3} \,\Bigg|\genfrac{}{}{0pt}{}{(c_{g+i}^A,\gamma_{g+i}^A)_{1,P'}}{{-\!-}\!{-\!-}} \right]H^{0,g}_{g,0} \left[ t \,\Bigg|\genfrac{}{}{0pt}{}{(v^A_i, \nu^A_i)_{1,g}}{{-\!-}\!{-\!-}} \right]\mathrm{d}t,
\end{eqnarray*}
such that $a^*_{H^{0,N}_{P,0}}=a^*_{H^{0,N'}_{P',0}}+\sigma^A_3 a^*_{H^{0,g}_{g,0}}$ with $\sigma^A_3,a^*_{H^{0,N'}_{P',0}}, a^*_{H^{0,g}_{g,0}}>0$, where $N':=N-g$ and $P':=P-g$.
\end{enumerate}

\item[$\mathbf{B}$] 
\begin{enumerate}
    \item We decompose a given $H^{m,n}_{p,q}$ with $n=M$ and $p=Q$ by using Theorem~2.9 in \cite{SaiKil}, identifying $(a_i,\alpha_i)_{1,p}=(1-d_i^B-\eta^B_1 \delta^B_i, \sigma^B_1 \delta^B_i)_{1,Q}$, such that assumption~(2.8.3) in \cite{SaiKil} holds, as follow
\begin{eqnarray*}
&&H^{m,M}_{Q,q} \left[ x \,\Bigg|\genfrac{}{}{0pt}{}{(1-d_i^B-\eta^B_1 \delta^B_i, \sigma^B_1 \delta^B_i)_{1,Q}}{(b_j,\beta_j)_{1,q}} \right]=\\
&&\int_0^\infty t^{\eta^B_1 -1} H^{m,0}_{0,q} \left[ x t^{\sigma^B_1} \,\Bigg|\genfrac{}{}{0pt}{}{{-\!-}\!{-\!-}}{(b_j,\beta_j)_{1,q}} \right]H^{M,0}_{0,Q} \left[ t \,\Bigg|\genfrac{}{}{0pt}{}{{-\!-}\!{-\!-}\!{-\!-}}{(d_i^B, \delta_i^B)_{1,Q}} \right]\mathrm{d}t,
\end{eqnarray*}
provided that $a^*_{H^{m,M}_{Q,q}}=a^*_{H^{m,0}_{0,q}} + \sigma^B_1 a^*_{H^{M,0}_{0,Q}}$ with $\sigma^B_1,a^*_{H^{m,0}_{0,q}}, a^*_{H^{M,0}_{0,Q}} >0$.\\
\item The functions $H^{m,0}_{0,q} \left[ x \,\big|\genfrac{}{}{0pt}{}{{-\!-}\!{-\!-}}{(b_j,\beta_j)_{1,q}} \right]$ and $H^{M,0}_{0,Q} \left[ x \,\big|\genfrac{}{}{0pt}{}{{-\!-}\!{-\!-}\!{-\!-}}{(d_i^B, \delta_i^B)_{1,Q}} \right]$ can be decomposed as in step $\mathbf{A}\text{-}(2)$.
\end{enumerate}
\end{description}
\begin{remark}
    \begin{enumerate}
        \item  In the case $a^*_{H^{m,0}_{0,q}}=0$, $\mu_{H^{m,0}_{0,q}}<-1$ or $ a^*_{H^{0,N}_{P,0}}=0$, $\mu_{H^{0,N}_{P,0}}<-1$, Corollary~2.10.1 in \cite{SaiKil} leads to the same integral decomposition of Corollary~2.9.1, where the parameter $\mu$ of a Fox-$H$ function is given in Equation~(1.1.10) in \cite{SaiKil}. Hence, the steps in $\mathbf{A}$ can also be stated for these cases.
        \item In the case $a^*_{H^{m,0}_{0,q}}=0$, $\mu_{H^{m,0}_{0,q}}<-1$ or $ a^*_{H^{M,0}_{0,Q}}=0$, $\mu_{H^{M,0}_{0,Q}}<-1$, step $\mathbf{B}\text{-}(1)$ can be performed by applying Theorem~2.10 in \cite{SaiKil}.
    \end{enumerate}
\end{remark}

\begin{definition}
    Let $H^{m,n}_{p,q}(x)$, $x>0$, be a given Fox-$H$ function. If $H^{m,n}_{p,q}(\cdot)$ can be decomposed by procedure $\mathbf{A}$, then we call $H^{m',0}_{0,q'}(\cdot)$, $H^{0,N'}_{P',0}(\cdot)$ its atomic functions and $H^{h,0}_{0,h}(\cdot)$, $H^{0,g}_{g,0}(\cdot)$ its auxiliary functions.
\end{definition}

\begin{theorem}\label{thm:FoxHPosi}
     Let $H^{m,n}_{p,q}(x)$, $x>0$, be a given Fox-$H$ function. If $H^{m,n}_{p,q}(\cdot)$ can be decomposed by procedure $\mathbf{A}$, its atomic and auxiliary functions are non-negative, then $H^{m,n}_{p,q}(\cdot)$ is non-negative on $(0,\infty)$.
\end{theorem}
\begin{proof}
    The proof is straightforward by noting that the procedure $\mathbf{A}$ decomposes $H^{m,n}_{p,q}$ into nested integrals, where the integrand functions are its atomic or auxiliary functions.
\end{proof}
\begin{remark}
    \begin{enumerate}
        \item The minimal assumptions on the parameters $a^*$ and $\mu$ guarantee the existence of the atomic or auxiliary functions, see Theorem~\ref{thm:ResidueDeltaPosi} and Theorem~1.1 in \cite{SaiKil};
        \item The positivity of the auxiliary functions $H^{h,0}_{0,h}(x)$, $x>0$, could be studied by applying Corollary 2.10.1 in \cite{SaiKil} to Equation~(2.9.4) in \cite{SaiKil} repeatedly;
        \item The positivity of the auxiliary functions $H^{0,g}_{g,0}(x)$, $x>0$, could be obtained by applying Equation~(2.1.3) and  Equation~(2.9.4) in \cite{SaiKil} to $H^{0,g}_{g,0}(x)$, $x>0$. 
        \item The sign of the atomic functions could be studied by their series representation, the method used in \cite{K2021} or using the function {\fontfamily{cmr}\selectfont FoxH[$\cdot$]} in Wolfram Alpha, see \cite{WolfAlph}. 
    \end{enumerate}
\end{remark}
\begin{remark}
Motivated by the work of Carter and Springer \cite{CarSprin}, we aim to develop a general framework that covers a wide class of (in general) non-Gaussian processes. This class is characterized by subordinations of Gaussian processes whose Laplace transform is represented by Fox-$H$ functions. In this way, the Fox-$H$ functions we are interested in correspond to the density of certain processes, see Lemma~\ref{lem:HFiniteDensities} below. Thus, assumption $(4)$ in Lemma~\ref{lem:AllHdensity} is automatically satisfied.

\end{remark}

\begin{remark}
		Other properties, such as infinite divisibility, atoms or asymptotic behavior, of a subclass of Fox-$H$ densities is questioned and investigated in \cite{Jan10,KSW20,Duf10,CL99,KarpPril16}.
\end{remark}

The following lemma gives the class of FHdam whose LT may be expressed by generalized Wright functions.

		\begin{lemma}[Fox-$H$  densities with all moments finite]\label{lem:HFiniteDensities}
			Let the assumptions of Lemma~\ref{lem:AllHdensity} hold for $n=0$ and $q=m$. Then the corresponding Fox-$H$ density	
			\begin{equation}\label{eq:densitiesFiniteMoments}  \varrho(x)=\frac{1}{K}H^{m,0}_{p,m} \left[ x \,\Bigg|\genfrac{}{}{0pt}{}{(a_i,\alpha_i)_{1,p}}{(b_j,\beta_j)_{1,m}} \right], \quad x>0,\end{equation}
			where \[ K=\mathcal{H}^{m,0}_{p,m}(1)=
			\frac{\prod_{j=1}^m \Gamma(b_j+\beta_j)}{\prod_{i=1}^p \Gamma(a_i+\alpha_i)}\] 
   has finite moments of all orders. The moments are given by 
			\[
			\int_{0}^\infty x^l  \varrho(x)\, \mathrm{d}x
			=\frac{1}{K}(\mathcal{M}H^{m,0}_{p,m})(l+1), \quad l=0,1,\dots,
			\]
   where $\mathcal{M}$ is the Mellin transform, i.e.~$(\mathcal{M}f)(s):=\int_{0}^{\infty}x^{s-1}f(x)\mathrm{d}x$; see Equation~(2.5.1) in \cite{SaiKil}. 

   Furthermore, if $ \sum_{i=1}^p \alpha_i - \sum_{j=1}^m \beta_j >-1$, then its LT is given by
			\begin{equation}\label{eq:LTdensitiesFiniteMoments} \phi(s):=\frac{1}{K}H^{1,m}_{m,p+1}\left[ s \,\bigg| \genfrac{}{}{0pt}{}{(1-b_j-\beta_j,\beta_j)_{1,m}}{(0,1),(1-a_i-\alpha_i,\alpha_i)_{1,p}}\right]=\frac{1}{K}\, _m\Psi_p\left[ \genfrac{}{}{0pt}{}{(b_j+\beta_j,\beta_j)_{1,m}}{(a_i+\alpha_i,\alpha_i)_{1,p}} \middle| -s\right],   \end{equation}
			where $s\geq 0$, has an entire extension.\\
        We denote by $\mathcal{X}$ the class of random variables with Fox-$H$ density $\varrho(\cdot)$ given in Equation~\eqref{eq:densitiesFiniteMoments}.
            
		\end{lemma}
		\begin{proof}
			For $n=0$ and $q=m$, the $\Delta$ and $a^*$ corresponding to $\varrho$ are such that $\Delta=a^*$. The Mellin transform of $\varrho$ is given by Theorem~2.2 in \cite{SaiKil}:
			
			\begin{equation*} 
				\int_{0}^\infty x^l  \varrho(x)\, \mathrm{d}x=\frac{1}{K}(\mathcal{M}H^{m,0}_{p,m})(l+1),
			\end{equation*}
			for all $l\in \mathbb{N}_0$.\\
  
    If $ \sum_{i=1}^p \alpha_i - \sum_{j=1}^m \beta_j >-1$, then the parameter $\Delta$ corresponding to $\phi(\cdot)$ is positive. Therefore, by applying Theorem~1.3 in \cite{SaiKil}, $\phi(\cdot)$ defined on $\mathcal{L}=\mathcal{L}_{-\infty}$ is analytic on $\C\backslash\{0\}$. As $\phi(\cdot)$ is a LT defined by choosing $\mathcal{L}=\mathcal{L}_{\mathrm{i} \gamma\infty}$, we use Theorem~\ref{thm:ResidueDeltaPosi} to extend analytically $\phi(\cdot)$ to $\C\backslash\{0\}$.
   Equation~\eqref{eq:LTdensitiesFiniteMoments} follows from Equation~(5.2) in \cite{KST02}.  \end{proof}

	\begin{remark}
	    In Theorem~1.5 in \cite{KilSvri} the authors give the condition on $\Delta$ for the absolute convergence of the generalized Wright functions defined through series.
	\end{remark}

  \begin{remark}
            Let $a^*_\varrho$ (resp.~$a^*_{\phi}$) be the parameter defined in Equation~(1.1.7) in \cite{SaiKil} associated with the Fox-$H$ density in Equation~\eqref{eq:densitiesFiniteMoments} (resp.~\eqref{eq:LTdensitiesFiniteMoments}). In addition, let $\Delta_\varrho$ (resp.~$\Delta_{\phi}$) be the parameter defined in Equation~(1.1.8) in \cite{SaiKil}  associated to the Fox-$H$ density in Equation~\eqref{eq:densitiesFiniteMoments} (resp.~\eqref{eq:LTdensitiesFiniteMoments}). Then we have the following relations: 
            $a^*_{\phi}=1+a^*_\varrho $ and $\Delta_{\phi}=1-\Delta_\varrho,$ 
            where $a^*_\varrho=\sum_{j=1}^m \beta_j - \sum_{i=1}^p \alpha_i=\Delta_\varrho$. For the other parameters $\mu_\varrho$, $\mu_\phi$, $\delta_\varrho$ and $\delta_\phi$ the following relationships hold: $\mu_\phi=\mu_\varrho + a^*_\varrho + m-p-1/2$
            and 
            $\delta_\phi \delta_\varrho=1$.
  \end{remark}

Now, we would like to emphasize an algebraic property of the class $\mathcal{X}$ of random variables with FHdam density: it is a multiplicative monoid with identity element given by the delta distribution at $1$, i.e.~$\delta_1(\cdot)$. This result is new to our knowledge.

\begin{lemma} \label{lem:ClosureLemma}
Let $(\Omega, \mathcal{F}, \mathbb{P})$ be a probability space, $X,Y \in \mathcal{X}$ be independent random variables, and $r_1, r_2 \in (0,\infty)$ be real positive numbers. Then the random variable $X^{r_1}Y^{r_2} \in \mathcal{X}$. Moreover, the subset $\mathcal{X}_{\mathrm{indep.}}\subset \mathcal{X}$ formed by independent random variables is a multiplicative monoid. The identity element of $(\mathcal{X}_{\mathrm{indep.}}, \cdot)$ is the random variable with distribution $H^{0,0}_{0,0}(\cdot)=\delta_1(\cdot)$.
\end{lemma}

\begin{proof}
Under the conditions of the lemma, using Theorem~\ref{thmpowertH0inf}, we have $X^{r_1},Y^{r_2} \in \mathcal{X}$. Moreover, if $X',Y' \in \mathcal{X}$, we also have $X' Y' \in \mathcal{X}$ by applying Theorem~\ref{thmproductH}. For the random variable with support $(0,1)$, Remark~\ref{rem:thmProductPowerHdensityon01}-$i)$ leads to the result.\\
Now we show that $H^{0,0}_{0,0}(\cdot)$ is the identity element. Using Definition~\ref{def:H-function}, with $\mathcal{L}$ being the imaginary axes, we obtain the following:
			
			\[
			H^{0,0}_{0,0}(x)=\frac{1}{2 \pi \mathrm{i}} \int^{\mathrm{i} \infty}_{-\mathrm{i} \infty} x^{-s}\,\mathrm{d}s=\frac{1}{2 \pi} \int_{\mathbb{R}}x^{-\mathrm{i}t}\,\mathrm{d}t, \quad x>0,
			\]
		where, in the second equality, we have used the change of variable $s=\mathrm{i}t$, for $t \in \mathbb{R}$. 
			Now it is easy to see that $H^{0,0}_{0,0}(\cdot)$ can be written as
			\[  H^{0,0}_{0,0}(x)=\frac{1}{2 \pi }\int_{\mathbb{R}} \mathrm{e}^{-\mathrm{i}\ln(x)t}\mathrm{d}t =\delta_0(-\ln(x)),\]
    where $\delta_{0}$ is the delta distribution at $0$.
			Thus, we have  
			\[ H^{0,0}_{0,0}(x)=\delta_0(-\ln(x))\equiv \delta_1(x), \quad x>0. \]
   We conclude by noting that $\delta_1(\cdot)$ is the distribution of a degenerate random variable almost surely equals one.
\end{proof}
\begin{remark}
	In \cite{SB04}, the authors show a result similar to $H^{0,0}_{0,0}(x)=\delta_1(x)$ but with a different proof.
\end{remark}

        \begin{corollary}\label{prop:AsymptotHDensity}
    Let the assumptions of Lemma~\ref{lem:HFiniteDensities} hold, with $a^*>0$ and the poles $b_{jl}$ in Equation \eqref{def:polesb} being simple, then we have that
    \begin{equation}\label{eq:AsymptoticHDensityinfinity} H^{m,0}_{p,m}(x)\sim O\left(x^{(\Re(\mu)+1/2)/\Delta} \exp(-\Delta \delta^{-1/\Delta} x^{1/\Delta})\right), \quad x \to +\infty    \end{equation}
    and
    \begin{equation}\label{eq:AsymptoticHDensityzero} H^{m,0}_{p,m}(x)\sim O\left(x^{\rho}\right), \quad x \to 0^+,   \end{equation}
    where 
    \[  \rho:=\min_{j=1,\dots,m}\left[ \frac{b_{j}}{\beta_{j}} \right]>-1. \]
        \end{corollary}
    \begin{proof}

    The asymptotic formula in Equation~\eqref{eq:AsymptoticHDensityinfinity} is a consequence of Corollary~1.10.2 in \cite{SaiKil} for the Fox-$H$ function with $m=q$, $n=0$ and $\Delta>0$, while the asymptotic formula in Equation~\eqref{eq:AsymptoticHDensityzero} holds from Corollary~1.11.1 in \cite{SaiKil}.
\end{proof}
 \begin{remark} For the asymptotic expansion at zero in the case that some poles $b_{jl}$ coincide, see Corollary~1.12.1 in \cite{SaiKil}.
\end{remark}

We conclude this section by giving two examples of applications of Lemma~\ref{lem:AllHdensity} and Lemma~\ref{lem:HFiniteDensities}. The first focuses on the upper-incomplete gamma function, $\Gamma(\rho,\cdot)$ with $\rho \in (0,1]$, (see \cite{DLMF} for details). Here, we show that $\Gamma(\rho,x):=\int_{x}^{\infty}w^{\rho -1}e^{-w}\mathrm{d}w$, $x \in (0,\infty)$, is a density, if properly normalized.
\begin{example}
	\label{exa:GammaIncompleteDensity}
	For $\rho \in (0,1]$, the upper-incomplete gamma function is obtained as the LT of a density $G^{0,1}_{1,1}(\cdot)$
	
	\[ \mathscr{L}\left(G^{0,1}_{1,1}\left[\cdot \Bigg|\genfrac{}{}{0pt}{}{-\rho}{-1} \right]\right)(s)=H^{2,0}_{1,2}\left[s \, \Bigg|\genfrac{}{}{0pt}{}{(1,1)}{(0,1),(\rho,1)}\right]=\Gamma(\rho,s),\quad  s \geq 0 ,\]
 see~\cite{BegGaj20}. 
	Thus, $\Gamma(\rho,s) \geq0$ for $s\geq 0$ and,
	by Lemma~\ref{lem:AllHdensity}, the following function
	
	\[\varrho(x)=\frac{1}{\Gamma(\rho+1)}H^{2,0}_{1,2}\left[x \, \Bigg|\genfrac{}{}{0pt}{}{(1,1)}{(0,1),(\rho,1)}\right], \quad x>0, \] 
	is a density.
\end{example}
As an application of Lemma~\ref{lem:HFiniteDensities}, we define the generalized $M$-Wright density (also called the generalized Mainardi density) by exploiting the $M$-Wright function $M_\beta$ (also known as Mainardi function), see \cite{MMP10} and \cite{Pag13}, for details.
 \begin{example}[generalized $M$-Wright density]\label{exa:threeparameterWrightdensity}
    Let $a \in \mathbb{R}$ and $\alpha>0$ be such that $a+\alpha>0$ and let $\beta \in (0,1)$. We define the density $\varrho(x)$, for $x \in (0,\infty)$, as follows 
			\begin{equation}\label{eq:gMWrightdensity}  \varrho(x):=\frac{\Gamma(1-\beta + \beta a +\beta \alpha)}{\Gamma(a+\alpha)} H^{1,0}_{1,1}\left[ x\, \Bigg| \genfrac{}{}{0pt}{}{(1-\beta+\beta a,\beta\alpha)}{(a, \alpha)} \right]. \end{equation}
			The fact that $\varrho(\cdot)$ is a density follows from Lemma~\ref{lem:HFiniteDensities}. Indeed, the only non-trivial condition from Lemma~\ref{lem:HFiniteDensities} to check is the non-negativity of $\varrho(\cdot)$. From Equation~(2.1.4) in \cite{SaiKil} and Equation~(1.59) in \cite{Mathai2009} we have that
			\[ \alpha H^{1,0}_{1,1}\left[ x \Bigg| \genfrac{}{}{0pt}{}{(1-\beta+\beta a,\beta\alpha)}{(a, \alpha)} \right]=x^{\frac{a}{\alpha}} H^{1,0}_{1,1}\left[ x^{\frac{1}{\alpha}} \Bigg| \genfrac{}{}{0pt}{}{(1-\beta,\beta)}{(0, 1)} \right] \quad x>0. \]
			Since the Fox-$H$ function on the right-hand side is the $M$-Wright function (see Equation~(4.9) in \cite{MPS05}), we have that the Fox-$H$ function on the left-hand side is non-negative. The density $\varrho(\cdot)$ defined in Equation~\eqref{eq:gMWrightdensity} is called generalized $M$-Wright density and the random variable associated with this density is called generalized $M$-Wright random variable.
  \end{example}

\section{Particular classes of Fox-$H$ densities with all moments finite}\label{sec:FoxHDensities}

		In this section, we present several classes of Fox-$H$ densities covered by the Lemma~\ref{lem:HFiniteDensities}.

\begin{theorem}\label{thm:Hdensity}
			
   Let $m,p,l \in \mathbb{N}_0$ with $e_j, a_k \in \R$, $\eta_j,\gamma_k,\alpha_k,b_i>0$, $\beta_k \in (0,1)$ such that $e_j+\eta_j>0$, $a_k+\alpha_k>0$ for $i=1,\dots,p$, $j=1,\dots,p+m$ and $k=1,\dots, l$.
			Then the following functions, $\varrho_\iota(\cdot)$, $\iota=0,\dots,7$, belong to the FHdam class.  
\begin{description}			
				\item[(C0)] For $m=p=0$,
				\begin{equation}\label{eq:H0000equalsdelta}
					\varrho_0(x):=\frac{1}{K_0}H^{0,0}_{0,0}\Big[ x\, \Big|\genfrac{}{}{0pt}{}{-\!-}{-\!-} \Big] =\delta_1(x), \quad x>0,
				\end{equation}
				where $K_0:=1$.
				\item[(C1)] For $m>0$, 
				
				\begin{equation*}
					\varrho_1(x):= \frac{1}{K_1} H^{m,0}_{0,m}\left[ x\, \Bigg|
					\genfrac{}{}{0pt}{}{{-\!-}\!{-\!-}\!{-\!-}}{(e_i, \eta_i)_{1,m}}\right], \quad x >0,
				\end{equation*}
				where $ K_1:= \prod_{i=1}^{m}\Gamma(e_i+\eta_i)$.
				
				\item[(C2)] For $p>0$, 
				
				\begin{equation*}
					\varrho_2(x):= \frac{1}{K_2} H^{p,0}_{p,p}\left[ x \Bigg|
					\genfrac{}{}{0pt}{}{(e_i+b_i,\eta_i)_{1,p}}{(e_i, \eta_i)_{1,p}}\right], \quad x \in (0,1),
				\end{equation*}
				where 
				$ K_2:= \prod_{i=1}^p \dfrac{\Gamma(e_i+\eta_i)}{\Gamma(e_i +b_i+\eta_i)}. $

				\item[(C3)] For $m,p>0$,
				\begin{equation*}\varrho_3(x):=\frac{1}{K_3} H^{m+p,0}_{p,m+p}\left[ x \, \Bigg|\genfrac{}{}{0pt}{}{(e_i+b_i,\eta_i)_{1,p}}{(e_i, \eta_i)_{1,p+m}} \right], \quad x>0,  \end{equation*}
				where 
				$ K_3:= \prod_{i=1}^p \dfrac{\Gamma(e_i+\eta_i)} {\Gamma(e_i +b_i+\eta_i)}\prod_{i=p+1}^{p+m}\Gamma(e_i+\eta_i). $
				
				\item[(C4)] For $l>0$,
				\[  \varrho_4(x):=\frac{1}{K_4}H^{l,0}_{l,l}\left[ x \Bigg| \genfrac{}{}{0pt}{}{(1 -\beta_k + \beta_kc_k,\alpha_k \beta_k\gamma_k)_{1,l}}{(c_k, \alpha_k\gamma_k)_{1,l}}\right], \quad x>0. \]
                 where $c_k:=a_k-\alpha_k\gamma_k+\alpha_k$ and
                $  K_4:=\prod_{k=1}^l \dfrac{\Gamma(a_k+\alpha_k)}{\Gamma(1-\beta_k+\beta_ka_k + \beta_k \alpha_k)}. $
				
				\item[(C5)] For $m>0$ and $l>0$, 
				
				\begin{equation*}
					\varrho_5(x):= \frac{1}{K_5} H^{m+l,0}_{l,m+l}\left[ x\, \Bigg|
					\genfrac{}{}{0pt}{}{(1 -\beta_k + \beta_kc_k,\alpha_k \beta_k\gamma_k)_{1,l}}{(e_i, \eta_i)_{1,m}(c_k, \alpha_k\gamma_k)_{1,l}}\right], \quad x >0,
				\end{equation*}
				where $c_k:=a_k-\alpha_k\gamma_k+\alpha_k$ and $K_5:=K_1K_4$.
				
				\item[(C6)] For $p>0$ and $l>0$, 
				
				\begin{equation*}
					\varrho_6(x):= \frac{1}{K_6} H^{p+l,0}_{p+l,p+l}\left[ x \Bigg|
					\genfrac{}{}{0pt}{}{(e_i+b_i,\eta_i)_{1,p}(1 -\beta_k + \beta_kc_k,\alpha_k \beta_k\gamma_k)_{1,l}}{(e_i, \eta_i)_{1,p}(c_k, \alpha_k\gamma_k)_{1,l}}\right], \quad x>0,
				\end{equation*}
				where $c_k=a_k-\alpha_k\gamma_k+\alpha_k$ and $K_6:=K_2K_4$.

				\item[(C7)] For $m >0$, $ p> 0$ and $l>0$,
				\begin{equation*}\varrho_7(x):= \frac{1}{K_7} H^{m+p+l,0}_{p+l,m+p+l}\left[ x \, \Bigg|\genfrac{}{}{0pt}{}{(e_i+b_i,\eta_i)_{1,p}(1 -\beta_k + \beta_kc_k,\alpha_k \beta_k\gamma_k)_{1,l}}{(e_i, \eta_i)_{1,p+m}(c_k, \alpha_k\gamma_k)_{1,l}} \right], \quad x>0,  \end{equation*}
				where $c_k=a_k-\alpha_k\gamma_k+\alpha_k$ and $K_7:=K_3K_4$. 
				
			\end{description}
   
		\end{theorem}

  \begin{remark}\label{rem:classes-rvariables}
Before proving Theorem~\ref{thm:Hdensity} we would like to give some details on the eight classes we found using Lemma~\ref{lem:ClosureLemma}. The main feature is that the classes (C1), (C2), and (C4) generate (C3), (C5)--(C7). More precisely: 
\begin{itemize}
\item The class (C0) corresponds to the degenerate random variable with distribution $\delta_1$, see also Remark \ref{rem:C0generalized} for further considerations.
\item The class (C1) describes the class of gamma random variables, their product and powers. 
\item The class (C2) is similar to (C1) for beta random variables.
\item The class (C3) is obtained as the product of classes (C1) and (C2).
\item The class (C4) is associated with generalized $M$-Wright random variables, its products, and powers; see Equation~\eqref{eq:gMWrightdensity}.
\item The class (C5) is obtained as the product of classes (C1) and (C4).
\item The class (C6) is obtained as the product of classes (C2) and (C4).
\item The class (C7) is obtained as the product of the classes (C3) and (C4), or, equivalently, as the product of the classes (C1), (C2), and (C4).
\end{itemize}
\end{remark}
		\begin{proof}[Proof of Theorem~\ref{thm:Hdensity}]
In all cases, we have to check the assumptions of Lemma~\ref{lem:HFiniteDensities}. In addition, it is easy to verify assumptions 1--3 in the Lemma~\ref{lem:AllHdensity} and it only remains to check the non-negativity of the functions $\varrho_\iota(\cdot)$, $\iota=0,\dots,7$. To this aim, we note that the Fox-$H$ functions appearing in (C1)--(C7) are densities of known random variables, as pointed out in Remark~\ref{rem:classes-rvariables}. The class (C0) corresponds to a degenerate random variable (i.e.~its "density" is a distribution, more precisely the delta function at $1$). For that reason, we specify the random variables that are behind each case in the rest of the proof.\\

    \noindent (C0) The result follows from Lemma~\ref{lem:ClosureLemma}.\\
			(C1) Let $m >0$, $e_i \in \R$, and $\eta_i>0$ be such that $e_i+\eta_i>0$, $i=1,\dots,m$. Moreover, let $X_i$, $i=1,\dots,m$, be independent random variables with gamma densities, i.e. 
			\[ \varrho_{X_i}(x)=\frac{1}{\Gamma(e_i+\eta_i)}x^{e_i+\eta_i-1}\mathrm{e}^{-x}=\frac{1}{\Gamma(e_i+\eta_i)}G^{1,0}_{0,1}\left[x\, \Bigg|\genfrac{}{}{0pt}{}{{-\!-}\!{-\!-}\!{-\!-}}{e_i+\eta_i-1}\right], \quad x> 0;
			\]
    see Equation~(2.9.4) in \cite{SaiKil}.
			By Theorem~\ref{thmpowertH0inf}, the r.v.'s $Y_i:=X_i^{\eta_i}$ have densities given by
			\[  \varrho_{Y_i}(y)=\frac{1}{\Gamma(e_i+\eta_i)}H^{1,0}_{0,1}\left[y\, \Bigg|\genfrac{}{}{0pt}{}{{-\!-}\!{-\!-}}{(e_i,\eta_i)} \right], \quad y > 0, \]
			for $i=1,\dots,m$.\\
			Using Theorem~\ref{thmproductH}, the r.v. $Y:=\prod_{i=1}^m Y_i$ has density
			\[  \varrho_1(x):=\prod_{i=1}^m \frac{1}{\Gamma(e_i+\eta_i)} H^{m,0}_{0,m}\left[x\, \Bigg|\genfrac{}{}{0pt}{}{{-\!-}\!{-\!-}\!{-\!-}}{(e_i,\eta_i)_{1,m}}\right], \quad x > 0.  \]
			(C2) Let $p>0$ be given and consider, for $i=1,\dots,p$, $e_i\in \R$, $b_i,\eta_i>0$ such that $e_i+\eta_i>0$, $i=1,\dots,p$, and $X_i$ independent random variables with beta densities:
			
			\[   \varrho_{X_i}(x)=\frac{1}{K_{i,2}}x^{e_i+\eta_i-1}(1-x)^{b_i-1}=\frac{1}{K_{i,2}}G^{1,0}_{1,1}\left[x \, \Bigg|\genfrac{}{}{0pt}{}{e_i+\eta_i+b_i-1}{e_i+\eta_i-1} \right],\quad x \in (0,1),  \]
		       where 
			\[K_{i,2}:=\frac{\Gamma(e_i+\eta_i)}{\Gamma(e_i+\eta_i+b_i)};\]
            see Equation~(2.9.6) in \cite{SaiKil}.
			By applying Theorem~\ref{thmpowertH0inf} and considering Remark~\ref{rem:thmProductPowerHdensityon01}-$i)$, the r.v.'s $Y_i=X_i^{\eta_i}$, $i=1,\dots,p$, have densities:
			
			\[ \varrho_{Y_i}(y)= \frac{1}{K_{i,2}}  H^{1,0}_{1,1}\left[y \,\Bigg|\genfrac{}{}{0pt}{}{(e_i+b_i,\eta_i)}{(e_i,\eta_i)}\right], \quad y \in (0,1).   \]
			By applying Theorem~\ref{thmproductH} to $Y:=\prod_{i=1}^p Y_i$ and considering Remark~\ref{rem:thmProductPowerHdensityon01}-$i)$, the density of the r.v.~$Y$ reads
			\[\varrho_2(x):=\frac{1}{K_2}  H^{p,0}_{p,p}\left[x \,\Bigg|\genfrac{}{}{0pt}{}{(e_i+b_i,\eta_i)_{1,p}}{(e_i,\eta_i)_{1,p}}\right], \quad x \in (0,1), \]
			where \[K_2=\prod_
			{i=1}^p K_{i,2}=\prod_{i=1}^p \frac{\Gamma(e_i+\eta_i)}{\Gamma(e_i+\eta_i +b_i)}.\]
			(C3) Let $p>0$ and $m>0$ be given and consider $e_j\in \R$, $b_i,\eta_j>0$, $i=1,\dots,p$, $j=1,\dots,p+m$, such that $e_j+\eta_j>0$,  for  $j=1,\dots,p+m$. Define $X:=X_1X_2$, where $X_1$ and $X_2$ are r.v.'s with densities $\varrho_1$ and $\varrho_2$, respectively. It turns out that $X$ has a Fox-$H$ density  $\varrho_3(\cdot)$ given by Theorem~\ref{thmproductH} and Remark~\ref{rem:thmProductPowerHdensityon01}-$i)$, as
			
			\[  \varrho_3(x):=\frac{1}{K_3} H^{p+m,0}_{p,p+m}\left[ x \, \Bigg|\genfrac{}{}{0pt}{}{(e_i+b_i,\eta_i)_{1,p}}{(e_i, \eta_i)_{1,p+m}} \right], \quad x \in (0,\infty),  \]
			where 
			\[ K_3= \prod_{i=1}^p \frac{\Gamma(e_i+\eta_i)}{\Gamma(e_i +b_i+\eta_i)} \prod_{i=p+1}^{p+m}\Gamma(e_i+\eta_i). \] 
			(C4) Let $l> 0$, $\gamma_k >0$, $X_k$ be given, $k=1,\dots,l$, where $X_k$ are independent generalized $M$-Wright r.v.'s with parameters $a_k,\alpha_k$ and $\beta_k$, see Equation \eqref{eq:gMWrightdensity}.
             We compute the density $\varrho_{Y}(\cdot)$ of the r.v.~$Y:=\prod_{k=1}^l X_k^{\gamma_k}$ by using Theorem~\ref{thmpowertH0inf} and Theorem~\ref{thmproductH}:
			
			\[  \varrho_4(x)=\frac{1}{K_4}H^{l,0}_{l,l}\left[ x \Bigg| \genfrac{}{}{0pt}{}{(1 -\beta_k + \beta_ka_k +\beta_k \alpha_k -\beta_k \alpha_k \gamma_k,\alpha_k \beta_k\gamma_k)_{1,l}}{(a_k-\alpha_k\gamma_k+\alpha_k, \alpha_k\gamma_k)_{1,l}}\right], \quad x>0, \]
   where 
   \[  K_4=\prod_{k=1}^l \frac{\Gamma(a_k+\alpha_k)}{\Gamma(1-\beta_k+\beta_ka_k + \beta_k \alpha_k)}. \]
			(C5)--(C7) Applying Theorem~\ref{thmproductH} to the product of r.v.'s in (C4) and (C1) (resp.~(C3)), we obtain (C5) (resp.~(C7)). Analogously, to obtain the density in the class (C6), we apply Theorem~\ref{thmproductH} to the product of r.v.'s in (C4) and (C3). This concludes the proof.
		\end{proof}
		
		\begin{remark} \label{rem:C0generalized} 
 				The class (C0) may be generalized as follows: for every $x_0>0$,
				\[H^{0,0}_{0,0}\left( \frac{x}{x_0}\right)\equiv \delta_{x_0}(x), \quad x>0,\] 
				by a similar proof as in Lemma~\ref{lem:ClosureLemma}. Furthermore, the class (C0) is also obtained as a limit of class (C1). In fact, if we choose $m=1$, $e_1=e>0$ and $\eta_1=\eta$ in (C1), then, by Equation~(2.1.10) in \cite{SaiKil}, the density $\varrho_1(\cdot)$ in Theorem~\ref{thm:Hdensity} converges in a weak sense to the delta distribution in one, that is, for every $x>0$, $\varrho_1(x) \to \delta_1(x)$, as $\eta\to0$.
				
		\end{remark}

\begin{remark}
   The densities of classes (C2) and (C3) were studied in \cite{Spri80} and \cite{MatSax69} with a slightly different choice of parameters, see Equations~(6.4.8) and (6.4.9) in \cite{Spri80}.
\end{remark}	

		\begin{remark} For $\iota=0,\dots,7$, the densities $\varrho_\iota(\cdot)$ may be extended from $(0,\infty)$ to $\mathbb{R}\backslash\{0\}$ by $\frac{1}{2}\varrho_\iota(|x/x_0|)$, $x_0 \in \mathbb{R}\backslash \{0\}$.  
		\end{remark}

	The following result proves the complete monotonicity and presents the explicit series representations of $\phi_{\iota}= \mathscr{L}(\varrho_\iota)$, $\iota=0,\dots,7$. The complete monotonicity easily follows by considering that the functions $\phi_\iota$ are LT of the densities $\varrho_\iota$ given in Theorem~\ref{thm:Hdensity}. 
 
 Recall that a function $f \in C^{\infty}((0,\infty))$ is completely monotone if and only if $(-1)^kf^{(k)}(y)\geq 0$, $y>0$, where $f^{(k)}$ is derivative of order $k$, for $k \in \N_{0}$, see \cite{Schil12}.

			\begin{corollary}\label{cor:CompMonoAllmom}
			Let $m,p,l \in \mathbb{N}_0$ with $e_j, a_k \in \R$, $b_i,\eta_j,\gamma_k,\alpha_k>0$, $\beta_k \in (0,1)$ be such that $e_j+\eta_j>0$ and $a_k+\alpha_k>0$ for $i=1,\dots,p$, $j=1,\dots,p+m$ and $k=1,\dots, l$ and $\sum_{i}^p b_i>1$.\\
			Then the functions $\phi_\iota(s)$, $\iota=0,\dots,7$, are completely monotone for $s \in [0,\infty)$.\\ Furthermore, if $\Delta_{\phi_\iota}>0$, they can be extended analytically to the entire complex plan. Hence, for $z \in \mathbb{C}$:

			\begin{description}
				
				\item[(C0)]	For $m=p=0$
				
				\[\phi_0(z)=\, _{0}\Psi_{0}\left[ \genfrac{}{}{0pt}{}{{-\!-}\!{-\!-}}{{-\!-}\!{-\!-}} \middle| -z \right]\\
					=\sum_{n \geq 0} \frac{(-z)^n}{n!}=\exp(-z).
				\]

				\item[(C1)] For $m>0$
				
				\[\phi_1(z)=\frac{1}{K_1}\, _{m}\Psi_{0}\left[ \genfrac{}{}{0pt}{}{(e_i+\eta_i,\eta_i)_{1,m}}{{-\!-}\!{-\!-}\!{-\!-}} \middle| -z\right]\\
					=\frac{1}{K_1}\sum_{n \geq 0} \prod_{i=1}^{m}\Gamma(e_i+\eta_i(1+n)) \frac{(-z)^n}{n!}.
				\]

				\item[(C2)] For $p>0$
				
				\begin{eqnarray*}\phi_2(z)&=&\frac{1}{K_2}\, _{p}\Psi_{p}\left[ \genfrac{}{}{0pt}{}{(e_i+\eta_i,\eta_i)_{1,p}}{(e_i+b_i+\eta_i,\eta_i)_{1,p}} \middle| -z\right]\\
					&=&\frac{1}{K_2}\sum_{n \geq 0} \frac{ \prod_{i=1}^{p}\Gamma(e_i+\eta_i(1+n))}{\prod_{i=1}^{p} \Gamma(e_i+b_i+\eta_i(1+n))} \frac{(-z)^n}{n!}.
				\end{eqnarray*}
				
				\item[(C3)] For $p>0$ and $m>0$ \begin{eqnarray*}\phi_3(z)&=&\frac{1}{K_3}\, _{m+p}\Psi_{p}\left[ \genfrac{}{}{0pt}{}{(e_i+\eta_i,\eta_i)_{1,m+p}}{(e_i+b_i+\eta_i,\eta_i)_{1,p}} \middle| -z\right]\\
					&=&\frac{1}{K_3}\sum_{n \geq 0} \frac{ \prod_{i=1}^{m+p}\Gamma(e_i+\eta_i(1+n))}{\prod_{i=1}^{p} \Gamma(e_i+b_i+\eta_i(1+n))} \frac{(-z)^n}{n!}.
				\end{eqnarray*}

				\item[(C4)] 
				For $l>0$\begin{eqnarray*}\phi_4(z)&=& \frac{1}{K_4}\, _{l}\Psi_{l}\left[ \genfrac{}{}{0pt}{}{(a_k+\alpha_k,\alpha_k\gamma_k)_{1,l}}{(1+\beta_k(a_k+\alpha_k-1),\alpha_k\beta_k\gamma_k)_{1,l}} \middle| -z\right]\\
					&=& \frac{1}{K_4}\sum_{n \geq 0} \frac{ \prod_{i=1}^l \Gamma(a_k+\alpha_k+\alpha_k\gamma_kn)}{\prod_{i=1}^l \Gamma(1+\beta_k(a_k+\alpha_k-1)+\alpha_k\beta_k\gamma_k n)} \frac{(-z)^n}{n!}.
				\end{eqnarray*}

				\item[(C5)] 
				
				For $m>0$ and $l>0$\begin{eqnarray*}\phi_5(z)&=& \frac{1}{K_5}\, _{m+l}\Psi_{l}\left[ \genfrac{}{}{0pt}{}{(e_i+\eta_i,\eta_i)_{1,m},(a_k+\alpha_k,\alpha_k\gamma_k)_{1,l}}{(1+\beta_k(a_k+\alpha_k-1),\alpha_k\beta_k\gamma_k)_{1,l}} \middle| -z\right]\\
					&=& \frac{1}{K_5}\sum_{n \geq 0} \frac{ \prod_{i=1}^{m}\Gamma(e_i+\eta_i(1+n)) \prod_{i=1}^l \Gamma(a_k+\alpha_k+\alpha_k\gamma_kn)}{\prod_{i=1}^l \Gamma(1+\beta_k(a_k+\alpha_k-1)+\alpha_k\beta_k\gamma_k n)} \frac{(-z)^n}{n!}.
				\end{eqnarray*}
				
				\item[(C6)] For $p>0$, $l>0$

    \[ \phi_6(z)= \frac{1}{K_6}\, _{p+l}\Psi_{p+l}\left[ \genfrac{}{}{0pt}{}{(e_i+\eta_i,\eta_i)_{1,p},(a_k+\alpha_k,\alpha_k\gamma_k)_{1,l}}{(e_i+b_i+\eta_i,\eta_i)_{1,p},(1+\beta_k(a_k+\alpha_k-1),\alpha_k\beta_k\gamma_k)_{1,l}} \middle| -z\right] \]\[
	= \frac{1}{K_6}\sum_{n \geq 0} \frac{ \prod_{i=1}^{p}\Gamma(e_i+\eta_i(1+n)) \prod_{i=1}^l \Gamma(a_k+\alpha_k+\alpha_k\gamma_kn)}{\prod_{i=1}^{p} \Gamma(e_i+b_i+\eta_i(1+n))\prod_{i=1}^l \Gamma(1+\beta_k(a_k+\alpha_k-1)+\alpha_k\beta_k\gamma_k n)} \frac{(-z)^n}{n!}.
				\]
				
				\item[(C7)] For $p>0$, $m>0$ and $l>0$\[ \phi_7(z)= \frac{1}{K_7}\, _{m+p+l}\Psi_{p+l}\left[ \genfrac{}{}{0pt}{}{(e_i+\eta_i,\eta_i)_{1,m+p},(a_k+\alpha_k,\alpha_k\gamma_k)_{1,l}}{(e_i+b_i+\eta_i,\eta_i)_{1,p},(1+\beta_k(a_k+\alpha_k-1),\alpha_k\beta_k\gamma_k)_{1,l}} \middle| -z\right]\]
					\[ =\frac{1}{K_7}\sum_{n \geq 0} \frac{ \prod_{i=1}^{m+p}\Gamma(e_i+\eta_i(1+n)) \prod_{i=1}^l \Gamma(a_k+\alpha_k+\alpha_k\gamma_kn)}{\prod_{i=1}^{p} \Gamma(e_i+b_i+\eta_i(1+n))\prod_{i=1}^l \Gamma(1+\beta_k(a_k+\alpha_k-1)+\alpha_k\beta_k\gamma_k n)} \frac{(-z)^n}{n!}.
				\]
			\end{description}

		\end{corollary}
		\begin{proof}
			(C0) If $m=p=0$, by means of Equation~\eqref{eq:H0000equalsdelta} we have that
			
			\[  \phi_0(s):=\int_0^\infty \mathrm{e}^{-x s} H^{0,0}_{0,0}(x)\,\mathrm{d}x=\int_0^\infty \mathrm{e}^{-x s} \delta_1(x) \,\mathrm{d}x=\mathrm{e}^{-s}, \quad s \geq 0.\]
			The claim follows from Equation~(2.9.4) in \cite{SaiKil}, with $b=0$, $\beta=1$, and by recalling Equation~\eqref{eq:LTdensitiesFiniteMoments}. \\
			(C1) For $m>0$, the parameters of this class satisfy Theorem~2.3 in \cite{SaiKil}, hence we use it in $*$ as follows
			\[ \phi_1(z):=\frac{1}{K_1}\mathscr{L}\left( H^{m,0}_{0,m}\left[\cdot\, \Bigg|\genfrac{}{}{0pt}{}{{-\!-}\!{-\!-}\!{-\!-}}{(e_i,\eta_i)_{1,m}}\right] \right)(z)
			\overset{*}{=}\frac{1}{K_1}H^{1,m}_{m,1}\left[ z \, \Bigg| \genfrac{}{}{0pt}{}{(1-e_i-\eta_i, \eta_i)_{1,m}}{(0,1) } \right] \]
			for $z \in \mathbb{C}$, with $\Re(z)>0$. By continuity, its value at $z=0$ is equal to $\prod_{i=1}^m \Gamma(e_i+\eta_i)$, which coincides with the value of the normalization of the density of class (C1) and the result follows by Equation~\eqref{eq:LTdensitiesFiniteMoments}. \\
			(C2) For $p>0$ and $\sum_{i=1}^p b_i>1$, the density of the class (C2) in Theorem~\ref{thm:Hdensity} satisfies the hypothesis of Theorem~2.3 in \cite{SaiKil}, so that  \begin{eqnarray*} \phi_2(s)&:=& \frac{1}{K_2}\mathscr{L}\left(H^{p,0}_{p,p}\left[ \cdot \, \Bigg|
				\genfrac{}{}{0pt}{}{(e_i+b_i,\eta_i)_{1,p}}{(e_i, \eta_i)_{1,p}}\right]\right)(s)\\
				&=& \frac{1}{K_2}H^{1,p}_{p,p+1}\left[ s \, \Bigg| \genfrac{}{}{0pt}{}{(1-e_i-\eta_i, \eta_i)_{1,p}}{(0,1),(1-e_{i-1}-b_{i-1}-\eta_{i-1},\eta_{i-1})_{2,p+1} } \right]
			\end{eqnarray*} for $s>0$ and 
			\[ \phi_2(0)=\lim_{s \to 0}\phi_2(s)=1.\]
			An argument analogous to the one used for class (C1) gives the result, for $s \geq 0$.\\
			(C3)--(C7) Similar to the previous ones.\\
   
			Now we extend analytically the above functions $\phi_{\iota}$, $\iota=0,\dots,7$. 

			If $\Delta_{\phi_{\iota}}>0$, $\iota=1,\dots,7$, we apply Theorem~\ref{thm:ResidueDeltaPosi} to extend the domain of the functions $\phi_{\iota}(z)$ from the sector $|\arg(z)|<a^*\pi/2$ to $\mathbb{C}$. By applying Theorem~1.3 in \cite{SaiKil}, we have the series representation in $z \in \mathbb{C}$, which completes the proof.
		\end{proof}

  \begin{remark}
     The generalized Wright function in class (C4) with $l=2$, $a_1+\alpha_1=\gamma/2+1$, $\alpha_1\gamma_1=\alpha$, $\beta_1=2$, $a_2+\alpha_2=1$ $\alpha_2\gamma_2=1$ and $\beta_2=\alpha$ is the eigenfunction of the fractional Bessel derivative $\mathcal{B}^{\alpha}_{\gamma,0}$ with $\alpha \in (0,1/2)$; see Section 5 in \cite{DLS21}.
  \end{remark}

	The following two tables (Table \ref{tab:astarDelta} and Table \ref{tab:mudelta}) explicitly show the parameters $a^*_{\phi_\iota},\Delta_{\phi_\iota},\mu_{\phi_\iota}$ and $\delta_{\phi_\iota}$ for $\iota=0,\dots,7$.

\begin{table}[ht]
\caption{ The parameters $a^{*}_{\phi_\iota}$ and $\Delta_{\phi_\iota}$, $\iota=0,\dots,7$, associated to the LT of the eight classes}\label{tab:astarDelta}

			\begin{tabular}{@{}llll@{}}
					\toprule
				\backslashbox{$\iota$}{Parameter} & $a_{\phi_{\iota}}^{*}$ & $\Delta_{\phi_{\iota}}$\\
				\hline
				$0$ & ${\displaystyle 1}$ & ${\displaystyle 1}$\\ $1$ & ${\displaystyle 1+\sum_{i=1}^{m}\eta_{i}}$ & ${\displaystyle 1-\sum_{i=1}^{m}\eta_{i}}$ \\ $2$ & $1$ & $1$\\ $3$ & ${\displaystyle 1+\sum_{i=p+1}^{p+m}\eta_{i}}$ & ${\displaystyle 1-\sum_{i=p+1}^{p+m}\eta_{i}}$ \\ $4$ & ${\displaystyle 1+\sum_{k=1}^{l}\alpha_{k}\gamma_{k}(1-\beta_{k})}$ & ${\displaystyle 1-\sum_{k=1}^{l}\alpha_{k}\gamma_k(1-\beta_{k})}$\\ $5$ & ${\displaystyle 1+\sum_{i=1}^{m}\eta_{i}+\sum_{k=1}^{l}\alpha_{k}\gamma_{k}(1-\beta_{k})}$ & ${\displaystyle 1-\sum_{i=1}^{m}\eta_{i}-\sum_{k=1}^{l}\alpha_{k}\gamma_{k}(1-\beta_{k})}$\\ $6$ & ${\displaystyle 1+\sum_{k=1}^{l}\alpha_{k}\gamma_{k}(1-\beta_{k})}$ & ${\displaystyle 1-\sum_{k=1}^{l}\alpha_{k}\gamma_{k}(1-\beta_{k})}$\\ $7$ & ${\displaystyle 1+\sum_{i=p+1}^{p+m}\eta_{i}+\sum_{k=1}^{l}\alpha_{k}\gamma_{k}(1-\beta_{k})}$ & ${\displaystyle 1-\sum_{i=p+1}^{p+m}\eta_{i}-\sum_{k=1}^{l}\alpha_{k}\gamma_{k}(1-\beta_{k})}$\\
				\botrule
			\end{tabular}
		
		\end{table}
		
		\afterpage{
			\clearpage
			\centering

\begin{sidewaystable}
\captionof{table}{The parameters $\mu_{\phi_\iota}$ and $\delta_{\phi_\iota}$, $\iota=0,\dots,7$, associated with the LT of the eight classes}	\label{tab:mudelta}		
					\begin{tabular}{@{}llll@{}}
					\toprule
					\backslashbox{$\iota$}{Par.} & $\mu_{\phi_{\iota}}$ & $\delta_{\phi_{\iota}}$\\
					\hline
					$0$ & ${\displaystyle -\frac{1}{2}}$ & ${\displaystyle 1}$\\ $1$ & ${\displaystyle \sum_{i=1}^{m}(e_{i}+\eta_{i})-\frac{m+1}{2}}$ & ${\displaystyle \prod_{i=1}^{m}\eta_{i}^{-\eta_{i}}}$\\ $2$ & ${\displaystyle \frac{1}{2}-\sum_{j=1}^{p}b_{i}}$ & $1$\\ $3$ & ${\displaystyle \frac{m-1}{2}-\sum_{j=1}^{p}b_{i}-\sum_{i=p+1}^{p+m}\left(1-e_{i}-\eta_{i}\right)}$ & ${\displaystyle \prod_{i=p+1}^{p+m}\eta_{i}^{-\eta_{i}}}$\\ $4$ & ${\displaystyle \sum_{k=1}^l(\beta_k-1)(1-a_k-\alpha_k)-\frac{1}{2}}$ & ${\displaystyle \prod_{i=1}^{l}(\gamma_k\alpha_k)^{\alpha_k \gamma_k(\beta_{k}-1)}\prod_{i=1}^{l}\beta_{k}^{\alpha_k \beta_k \gamma_k}}$\\ $5$ & ${\displaystyle \sum_{i=1}^{m}(e_{i}+\eta_{i})+\sum_{k=1}^l(\beta_k-1)(1-a_k-\alpha_k)-\frac{m-1}{2}}$ & ${\displaystyle \prod_{i=1}^{m}\eta_{i}^{-\eta_{i}}\prod_{i=1}^{l}(\gamma_k\alpha_k)^{\alpha_k \gamma_k(\beta_{k}-1)}\prod_{i=1}^{l}\beta_{k}^{\alpha_k \beta_k \gamma_k}}$\\ $6$ & ${\displaystyle \sum_{k=1}^l(\beta_k-1)(1-a_k-\alpha_k)-\sum_{j=1}^{p}b_{i}}-\frac{1}{2}$ & ${\displaystyle \prod_{i=1}^{l}(\gamma_k\alpha_k)^{\alpha_k \gamma_k(\beta_{k}-1)}\prod_{i=1}^{l}\beta_{k}^{\alpha_k \beta_k \gamma_k}}$\\ $7$ & ${\displaystyle \sum_{k=1}^l(\beta_k-1)(1-a_k-\alpha_k)-\sum_{j=1}^{p}b_{i}-\sum_{i=p+1}^{p+m}\left(1-e_{i}-\eta_{i}\right)+\frac{m-1}{2}}$ & ${\displaystyle \prod_{i=p+1}^{p+m}\eta_{i}^{-\eta_{i}}\prod_{i=1}^{l}(\gamma_k\alpha_k)^{\alpha_k \gamma_k(\beta_{k}-1)}\prod_{i=1}^{l}\beta_{k}^{\alpha_k \beta_k \gamma_k}}$\\
			\botrule
			\end{tabular}

			\end{sidewaystable}
			\clearpage
		}
		
		\newpage

\section{Examples of elementary and special functions}\label{sec:Examples}
We present here examples of well-known special functions that can be obtained as special cases of our Fox-$H$ densities, which by Lemma~\ref{lem:HFiniteDensities} are generalized Wright functions.

\subsection{Examples in class (C0)}
\begin{description}
\item[Exponential function]
If $\varrho(\cdot)$ is in class (C0), then its LT is $ H^{1,0}_{0,1} \left[ s \Big|\genfrac{}{}{0pt}{}{-\!-}{(0,1)} \right] =\mathrm{e}^{-s }$ for $s \geq 0.$
\end{description}

\subsection{Examples in class (C1)}

\begin{description}
\item[Stretched gamma]
If $\varrho(\cdot)$ is in (C1) with $m=1$, $e\in \R$ and $\eta>0$ such that $e+\eta>0$,  then
\[  \varrho(x)=\frac{1}{\Gamma(e+\eta)} H^{1,0}_{0,1}\left[ x\Bigg|\genfrac{}{}{0pt}{}{-\!-}{(e,\eta)} \right]=\frac{1}{\Gamma(e+\eta)\eta}x^{e/\eta}\mathrm{e}^{-x^{1/\eta}}, \quad x \in (0,\infty), \]
see Equation~(2.9.4) in \cite{SaiKil}. For $\eta=1$ its LT is equal to
\[  \phi(s)=\frac{1}{\Gamma(b+\beta)}\, _1\Psi_0\left[\genfrac{}{}{0pt}{}{(e,1)}{-\!-} \Bigg\vert -s \right]= \, _1F_0(e+1;-s)=(1+s)^{-e-1}\]
where $_1F_0(\dots;\cdot)$ is a general hypergeometric function; see Equation~(2.9.3) and (2.9.5) in \cite{SaiKil}.\bigskip{}

\item[Modified Bessel function of the third kind]
If $\varrho(\cdot)$ is in (C1) with $m=2$ and $e_1=-\nu/2$ and $e_2=\nu/2$ and $\eta_1=\eta_2=1/2$ for $ \nu \in (-1,1)\backslash\{0\}$  then, for $x \in (0,\infty)$,
\[  \varrho(x)=\frac{1}{\Gamma((1-\nu)/2)\Gamma((1+\nu)/2))}H^{2,0}_{0,2}\left[ x\Bigg|\genfrac{}{}{0pt}{}{{-\!-}\!{-\!-}\!{-\!-}\!{-\!-}\!{-\!-}\!{-\!-}\!{-\!-}}{(-\nu/2,1/2),(\nu/2,1/2)} \right]=\frac{4K_{\nu}(2x)}{\Gamma((1-\nu)/2)\Gamma((1+\nu)/2))}, \]
where $K_{\nu}(\cdot)$ is the modified Bessel function of the third kind; see Equation~(2.9.19) in \cite{SaiKil} or Equation~(1.128) in \cite{Mathai2009}, for the definition.

\item[Airy function]

If $\varrho(\cdot)$ is in (C1) with $m=2$, $\eta_1=\eta_2=1$, $e_1=0$ and $e_2=1/3$ we have 

\[ \varrho(x)=\frac{1}{\Gamma(4/3)}G^{2,0}_{0,2}\left[ x\Bigg|\genfrac{}{}{0pt}{}{-\!-}{(0),(1/3)} \right]=\frac{2 \pi \sqrt[6]{3}}{\Gamma(4/3)}\mathrm{Ai}(\sqrt[3]{9x}), \quad x \in (0,\infty), \]

where $\mathrm{Ai}(\cdot)$ is the Airy function, see Equation~8.4.29-1 in \cite{PBM90III}.

\end{description}

\subsection{Examples in class (C2)}

\begin{description}

\item[Beta and confluent hypergeometric Kummer function]
If $\varrho(\cdot)$ is in (C2) with $p=1$, $\eta=1$, $e>-1$, $b>1$, then $\varrho(\cdot)$ is a beta density, see Equation~(2.9.6) in \cite{SaiKil}. Its LT reads:
\[ \phi(s)=\frac{\Gamma(e+1+b)}{\Gamma(e+1)}\, _1 \Psi_1 \left[ \genfrac{}{}{0pt}{}{(e+1,1)}{(e+b+1,1)}\bigg| \, -s\right]=\,_1F_1(e+1,e+b+1;-s), \quad s \geq 0, \]
where $_1F_1(e+1,e+b+1;\cdot)$ is the confluent hypergeometric Kummer function, as can be checked by using Equation~\eqref{eq:LTdensitiesFiniteMoments} and Equation~(2.9.14) in \cite{SaiKil}.

\item[Elementary functions]
If $\varrho(\cdot)$ is in Class (C2) with $p=2$, $\eta_1=\eta_2=b_2=1, e_1=e_2=0, b_1=1/2$ we have that  

\[  \varrho(x)=\Gamma\left(\frac{3}{2}\right)G^{2,0}_{2,2}\left[ x\Bigg|\genfrac{}{}{0pt}{}{(1/2), (1)}{(0),(0)} \right]=\log\left( \frac{1+\sqrt{1-x}}{\sqrt{x}}\right), \quad x\in(0,1). \]

Otherwise, if $p=2$ with $\eta_1=\eta_2=b_1=1$, $e_1=0$, $b_2=e_2=1/2$ we get 
\[ \varrho(x)=\Gamma\left(\frac{3}{2}\right)G^{2,0}_{2,2}\left[ x\Bigg|\genfrac{}{}{0pt}{}{(1), (1)}{(0),(1/2)} \right]=\arccos(\sqrt{x}), \quad x \in (0,1); \]
see Equations 8.4.6-22 and 8.4.7-1 in \cite{PBM90III} for representations of elementary functions through the $G$-functions.

\item[Kabe function]
If $\varrho(\cdot)$ is in class (C2) with $p=m$, $e_i=\gamma_i+1$, $b_i=\delta_i>0$, $\eta_i=1$ such that $\gamma_i>-2$ for $i=1,\dots,m$, then

\begin{equation}\label{eq:DensiKiryako} \varrho(x)=\frac{\Gamma(\delta_i+\gamma_i+2)}{\Gamma(\gamma_i+2)}G^{m,0}_{m,m}\left[ x\Bigg|\genfrac{}{}{0pt}{}{(\gamma_i+\delta_i+1)_{1,m}}{(\gamma_i+1)_{1,m}} \right], \quad x \in (0,1).\end{equation} 

The $G$-function that appears in Equation~\eqref{eq:DensiKiryako}, coincides with the integral kernel used for multiple Erderlélyi-Kober operators in Definition 1.1.1 in \cite{Kir}. Functions of type $G^{m,0}_{m,m}$ appear for the first time in \cite{Kab58} as a distribution for statistical applications.

\end{description}

\subsection{Examples in class (C3)}

\begin{description}

\item[Exponential integral function]
If $\varrho(\cdot)$ is in (C3) with $m=1$, $p=1$, $e_1=e_2=0$, $b_1=1$ and $\eta_1=\eta_2=1$, we have that

\[   \varrho(x)=G^{2,0}_{1,2}\left[ x\Bigg|\genfrac{}{}{0pt}{}{(1)}{(0),(0)} \right]=-\mathrm{Ei}(-x), \quad x>0,\]
where $\mathrm{Ei}(\cdot)$ is the exponential integral function; see  Equation~(25) of Section~(6.9.2) in \cite{ErdI} and Equation~8.4.11-1 in \cite{PBM90III}.

\item[Incomplete gamma]

If $\varrho(\cdot)$ is in (C3) with $p=1$, $m=1$ , $\eta_1=\eta_2=1$, $e_1=0$, $e_2=\rho$ and $b_1=1$ such that $\rho>-1$, then we have that

\[ \varrho(x)=\frac{1}{\Gamma(\rho+1)}G^{2,0}_{1,2}\left[x \, \Bigg|\genfrac{}{}{0pt}{}{(1)}{(0),(\rho)}\right]=\Gamma(\rho,x) \text{ for } x \geq 0 ,\]

see Equation~8.4.16-2 in \cite{PBM90III}.

\item[Gauss hypergeometric function and general Kober operators]
If $\varrho(\cdot)$ is in (C3) with $m=1$, $p=1$, $e_1=\alpha$, $e_2=\beta+n$, $b_1=\gamma-\alpha$ and $\eta_1=\eta_2=1$ with $\alpha,\beta+n>-1$ and $\gamma>\alpha+1$, we see that its LT for $s \geq0$ is given by

\[ \phi(s)=\frac{\Gamma(\gamma+1)}{\Gamma(\beta+n+1)\Gamma(\alpha+1)}G^{1,2}_{2,2}\left[ s \, \Bigg| \genfrac{}{}{0pt}{}{(1-\alpha),(1-\beta-n)}{(0),(1-\gamma)} \right]=\, _2F_1(\alpha,\beta+n;\gamma;-s)  \]
where $_2F_1(\dots;\cdot)$ is the Gauss hypergeometric function involved in generalized Kober operators; see Definition 3.14 in \cite{Mathai2009}.

\end{description}
\subsection{Examples in class (C4)}

\begin{description}

\item[Mittag-Leffler function]
If $\varrho(\cdot)$ is in class (C4) with $l=1$, $\alpha\gamma=1$, $\alpha+a=1$, and $\beta\in (0,1)$, then
\[ \varrho(x)=H^{1,0}_{1,1}\left[ x\Bigg|\genfrac{}{}{0pt}{}{(1-\beta,\beta)}{(0,1)} \right]= M_\beta(x), \quad x \in (0,\infty), 
 \]
where $M_\beta(\cdot)$ is the $M$-Wright function, see \cite{MPS05}.
Its LT is equal to

\[  \phi(s)=H^{1,1}_{1,2}\left[ s\, \Bigg|\genfrac{}{}{0pt}{}{(0,1)}{(0,1), (0,\beta)} \right]=\mathrm{E}_\beta(-s), \quad s\geq 0 ,
\]
where $\mathrm{E}_\beta(\cdot)$ is the Mittag-Leffler function, see \cite{GKMR20}.

\item[Three parameters Mittag-Leffler function]\label{exa:threeparamMittagLeffler}

If $\varrho(\cdot)$ is in class (C4), we choose $l=1$, $A=\beta$, $C=a+\alpha$, $B=1-\beta(1-a-\alpha)$ and $\alpha\gamma=1$ with $a+\alpha>0$, $\beta \in (0,1)$ and $\gamma>0$, so that its LT reads

\[  \phi(s)=\frac{\Gamma(B)}{\Gamma(C)}H^{1,1}_{1,2}\left[ s\, \Bigg|\genfrac{}{}{0pt}{}{(1-C,1)}{(0,1), (1-B,A)} \right]=\Gamma(B)\mathrm{E}^C_{A,B}(-s), \quad s \geq 0, \]
where $E^{C}_{A,B}$ is the three parameters Mittag-Leffler function; see Equation~(6.10) in \cite{KST02}. We note that if $\beta=A=1$, then $\varrho(\cdot)$ is in class (C0). 

\end{description}

Many other examples can be found in \cite{PBM90}, \cite{PBM90III}, \cite{ErdI} and \cite{VellKat}.  In \cite{Mathai2009}, \cite{MS78} and \cite{MatSax} we can find other $H$-functions of the form as $\varrho(\cdot)$ involved in applications.

\section{Conclusions}\label{Sec:Outlook}

We have identified a large class of Fox-$H$ densities and characterized their LT through generalized Wright functions. We give examples covered by the mentioned class of densities. In particular, the beta function, modified Bessel function of the third kind, upper incomplete gamma function, and Mittag-Leffler functions, see Section~\ref{sec:Examples}. 
The above setup is the background for further research directions, including a new class of densities, more refined algebraic structures, and statistical applications; see also \cite{Cook}, \cite{BC87}, \cite{Kab58} and \cite{Spri80}. Its generalization to infinite dimensions, as an example of non-Gaussian analysis, and its applications to various branches of mathematical physics will be further investigated. Through non-Gaussian analysis, we can study certain kinds of semi-Markovian diffusive processes, their time-changed representations, and their generalized functionals, e.g.~Donsker's delta function. 
In view of these further applications, the properties of the classes of functions analyzed here, such as complete monotonicity, appear to be crucial.

\newpage
\appendix
\section{Asymptotic behavior and analytic continuation}\label{sec:AppendixAnalyticalresults}

In this appendix, we start by recalling a well-known asymptotic behavior of the gamma function; then we prove Lemma~\ref{lem:asympHCircle} and Theorem~\ref{thm:ResidueDeltaPosi} stated in Section~\ref{sec:DefinitionExistenceAnaliticalContinuation}.

\subsection{Asymptotic behavior}

It is well-known that 
\[\Gamma(b+as)\sim \sqrt{2\pi} \mathrm{e}^{-as}(as)^{as+b-1/2}\; \mathrm{as}\; |s|\to \infty \] 
which can be written as  
\begin{equation}\label{asymp:Gamma+}
	\begin{split}
		|\Gamma(b+as)|\sim &\sqrt{2\pi}\mathrm{e}^{-\Im(b)}a^{a\Re(s)+\Re(b)-1/2}\\
		&\qquad \cdot \mathrm{e}^{(a\Re(s)+\Re(b)-1/2)\mathrm{Log}(|s|)}\\
		&\qquad \qquad \cdot \mathrm{e}^{-a(\Re(s)+\Im(s)\arg(s))},
	\end{split}
\end{equation}
as $|s|\to \infty$ and $|\arg(b+as)|<\pi$, where $a>0$, $b \in \mathbb{C}$; see Equation~(5.11.7) in \cite{AskRoy10}. \\
In order to express the asymptotic formula of $\mathcal{H}^{m,n}_{p,q}(s)z^{-s}$,
we also need the asymptotic behavior of $\Gamma(b-as)$, which may be derived from that of  $\Gamma(b+as)$, as follows.\\
Let $s$ be any complex number with polar coordinates $\rho=|s|> 0$, $\theta \in (-\pi,\pi]$, and denote $\tilde{s}:=-s$. Then

\[   \arg(\tilde{s})=\begin{cases} \theta+\pi, & \theta \in (-\pi,0] \\ \theta-\pi, & \theta \in (0,\pi] \end{cases}\]
and we compactly write $\arg(\tilde{s})=\mathrm{sign}(\arg(s))(|\arg(s)|-\pi)$.
Hence, $|\Gamma(b-as)|=|\Gamma(b+a\tilde{s})|$ and 

\begin{equation} \label{asymp:Gamma-}
	\begin{split}
		|\Gamma(b+a\tilde{s})|\sim & \sqrt{2\pi}\mathrm{e}^{-\Im(b)}a^{-a\Re(s)+\Re(b)-1/2}\\
		& \quad \cdot \mathrm{e}^{(-a\Re(s)+\Re(b)-1/2)\mathrm{Log}(|s|)}\\
		&\quad \quad \cdot \mathrm{e}^{a(\Re(s)+\Im(s)(\mathrm{sign}(\arg(s))(|\arg(s)|-\pi)))},
	\end{split}
\end{equation}
as $|s|\to \infty$ and $|\arg(b+a\tilde{s})|<\pi$. Note that, for $a>0$ and $b \in \mathbb{R}$, the poles of $\Gamma(b-as)$ are located in $[b/a,+\infty)$. Therefore, the above asymptotic formula is valid only for $s$ such that $|\arg(s)|>0$. The asymptotic formulas~\eqref{asymp:Gamma+} and \eqref{asymp:Gamma-} can be found in Equations~(2.4.4) in \cite{Par01}.

\begin{proof}[Proof of Lemma \ref{lem:asympHCircle}]
	From Definition~\ref{def:H-function}, we recall that for $s \in \mathcal{L}_{\mathrm{i}\gamma \infty}$ and $z \in \mathbb{C}\backslash\{0\}$,
	
	\begin{equation}\label{eq:KernelandFactors}
		\mathcal{H}^{m,n}_{p,q}(s)= \frac{\prod_{j=1}^{m} \Gamma(b_j+s\beta_j)}{\prod_{i=n+1}^{p}\Gamma(a_i+s\alpha_i)} \cdot \frac{\prod_{i=1}^{n}\Gamma(1-a_i-s\alpha_i)}{ \prod_{j=m+1}^{q} \Gamma(1-b_j-s\beta_j)},
	\end{equation}
	and that 
	\begin{equation}\label{eq:zpowers}
		z^{-s}=\mathrm{e}^{-s(\mathrm{Log}(|z|)+\mathrm{i} \arg(z))}.
	\end{equation}
	To obtain the result in Equation~\eqref{eq:HKernelonCircle}, we start by computing the asymptotic behavior of each factor in Equation~\eqref{eq:KernelandFactors}, then the behavior of $z^{-s}$ in Equation~\eqref{eq:zpowers}, for $|s|\to\infty$, and finally we put them together.
	\begin{enumerate}
		\item We compute the asymptotic behavior of 
		\[\frac{\prod_{j=1}^{m} \Gamma(b_j+s\beta_j)}{\prod_{i=n+1}^{p}\Gamma(a_i+s\alpha_i)}. \]
If $s=\gamma + R\mathrm{e}^{\mathrm{i}\theta}$, with $|\arg(\theta)|<\pi$, then $\Gamma(b_j+s\beta_j)=\Gamma(b_j+\beta_j\gamma+\beta_jR\mathrm{e}^{\mathrm{i}\theta})$, $j=1,\dots,m$, and $\Gamma(a_i+s\alpha_i)=\Gamma(a_i+\alpha_i\gamma+\alpha_iR\mathrm{e}^{\mathrm{i}\theta})$, $i=n+1,\dots,p$.

   By Equation~\eqref{asymp:Gamma+}, rewriting the asymptotic formula in such a way that the first factor is a constant (in $R$) and ordering the remaining factors from the leading to the weakest one, we have that	
	\begin{equation*}
			\begin{split}
				|\Gamma(b_j+\beta_j\gamma+\beta_jR\mathrm{e}^{\mathrm{i}\theta})| \sim &  \sqrt{2\pi}\beta_j^{b_j+\beta_j\gamma-1/2} \mathrm{e}^{\beta_j\cos(\theta)R\mathrm{Log}(|R|)}\\
				& \quad \quad \cdot  \mathrm{e}^{[\beta_j\cos(\theta)\mathrm{Log}(\beta_j)-\beta_j(\cos(\theta)+\sin(\theta)\theta)]R}\\
				& \quad \quad \quad \quad \cdot \mathrm{e}^{(b_j+\beta_j\gamma-1/2)\mathrm{Log}(|R|)},
			\end{split}
		\end{equation*}
		as $ R\to\infty$, for $|\arg(\theta)|<\pi$, and $j=1,\dots,m$.
		The product reads
	\begin{equation*}
			\begin{split}
				\prod_{j=1}^m |\Gamma(b_j+\beta_j\gamma+\beta_jR\mathrm{e}^{\mathrm{i}\theta})|&\sim  (2\pi)^{m/2}\prod_{j=1}^m\left(\beta_j^{b_j+\beta_j\gamma-1/2}\right) \mathrm{e}^{(\sum_{j=1}^m\beta_j)\cos(\theta)R\mathrm{Log}(|R|)} \\
				& \quad \quad \cdot \mathrm{e}^{\left[\cos(\theta)\mathrm{Log}\left(\prod_{j=1}^m\beta_j^{\beta_j}\right)-(\cos(\theta)+\sin(\theta)\theta)(\sum_{j=1}^m\beta_j)\right]R} \\
				& \quad \quad \quad \quad \cdot \mathrm{e}^{\left(\sum_{j=1}^m b_j+ \gamma\sum_{j=1}^m\beta_j-m/2\right)\mathrm{Log}(|R|)},
			\end{split}
		\end{equation*}
		as $ R\to\infty$, for $|\arg(\theta)|<\pi$. By analog computations for $\prod_{i=n+1}^p |\Gamma(a_i+\alpha_i\gamma+\alpha_iR\mathrm{e}^{\mathrm{i}\theta})|$, as $ R\to\infty$, for $|\arg(\theta)|<\pi$, we get the following asymptotic behavior for the quotient
		\begin{equation*}
			\begin{split}
				\frac{ \prod_{j=1}^m |\Gamma(b_j+\beta_j\gamma+\beta_jR\mathrm{e}^{\mathrm{i}\theta})|}{\prod_{i=n+1}^p |\Gamma(a_i+\alpha_i\gamma+\alpha_iR\mathrm{e}^{\mathrm{i}\theta})|}& \sim  (2\pi)^{(m+n-p)/2}\frac{\prod_{j=1}^m \beta_j^{b_j+\beta_j\gamma-1/2} }{\prod_{i=n+1}^p \alpha_i^{a_i+\alpha_i\gamma-1/2} } \mathrm{e}^{a_1^*\cos(\theta)R\mathrm{Log}(|R|)}  \\
				& \cdot  \mathrm{e}^{\left[ \cos(\theta)\mathrm{Log}\left(\prod_{j=1}^m\beta_j^{\beta_j} \prod_{i=n+1}^p\alpha_i^{-\alpha_i}\right)-\left(\cos(\theta)+\sin(\theta)\theta\right)a_1^* \right]R} \\
				& \cdot \mathrm{e}^{\left(\sum_{j=1}^m b_j-\sum_{i=n+1}^p a_i + \gamma a_1^*-(m-p+n)/2 \right)\mathrm{Log}(|R|)},\\
			\end{split}
		\end{equation*}
  as $R \to \infty$, where $a_1^*=\sum_{j=1}^m\beta_j-\sum_{i=n+1}^p\alpha_i$; see Equation~(1.1.11) in \cite{SaiKil}.
		\item Now we study the asymptotic behavior of the second factor, i.e.
		\[ \frac{\prod_{i=1}^{n}\Gamma(1-a_i-s\alpha_i)}{ \prod_{j=m+1}^{q} \Gamma(1-b_j-s\beta_j)}.\]		
By Equation~\eqref{asymp:Gamma-}, for $i=1,\dots,n$, we may rewrite
  \begin{equation*}
			\begin{split}
				 |\Gamma(1-a_i-\alpha_i \gamma - \alpha_i R \mathrm{e}^{\mathrm{i}\theta})|  \sim & \sqrt{2\pi} \alpha_i^{1/2-a_i-\alpha_i \gamma} \mathrm{e}^{- \alpha_i \cos(\theta)R\mathrm{Log}(|R|) } \\
				& \quad \quad \cdot  \mathrm{e}^{\left[ \cos(\theta)\mathrm{Log}\left(\alpha_i^{-\alpha_i}\right)+(\cos(\theta)+\sin(\theta)\mathrm{sign}(\theta)(|\theta|-\pi))\alpha_i\right] R} \\
				& \quad \quad \quad \quad \cdot  \mathrm{e}^{\left(1/2- a_i - \gamma \alpha_i \right)\mathrm{Log}(|R|)} ,\\
			\end{split}
		\end{equation*}
as $R \to \infty$ for $|\arg(z)|>0$.
		Consequently, the product reads 
		\begin{equation*}
			\begin{split}
				\prod_{i=1}^n |\Gamma(1-a_i-\alpha_i \gamma - \alpha_i R \mathrm{e}^{\mathrm{i}\theta})|  \sim & (2\pi)^{n/2} \prod_{i=1}^n \alpha_i^{1/2-a_i-\alpha_i \gamma}  \mathrm{e}^{-R \cos(\theta)\mathrm{Log}(|R|)\sum_{i=1}^n \alpha_i } \\
				&  \cdot  \mathrm{e}^{\left[\cos(\theta)\mathrm{Log}\left(\prod_{i=1}^n\alpha_i^{-\alpha_i}\right)+(\cos(\theta)+\sin(\theta)\mathrm{sign}(\theta)(|\theta|-\pi))\sum_{i=1}^n\alpha_i\right] R} \\
				& \quad \quad \quad \quad \cdot  \mathrm{e}^{\left(-\sum_{i=1}^n a_i - \gamma\sum_{i=1}^n \alpha_i +n/2\right)\mathrm{Log}(|R|)}, \\
			\end{split}
		\end{equation*}
  as $R \to \infty$ for $|\arg(z)|>0$.
		We omit the details for $\prod_{j=m+1}^q |\Gamma(1-b_j-\beta_j \gamma - \beta_j R \mathrm{e}^{\mathrm{i}\theta})|$, for $R \to \infty$ and $|\arg(z)|>0$, as follows by similar considerations. Finally, the asymptotic behavior of the second factor, as $R \to \infty$, for $|\arg(z)|>0$, is given by 
		\begin{equation*}
			\begin{split}
				\frac{\prod_{i=1}^n |\Gamma(1-a_i-\alpha_i \gamma - \alpha_i R \mathrm{e}^{\mathrm{i}\theta})|}{\prod_{j=m+1}^q |\Gamma(1-b_j-\beta_j \gamma - \beta_j R \mathrm{e}^{\mathrm{i}\theta})|} & \sim   (2\pi)^{(n+m-q)/2} \prod_{i=1}^n \alpha_i^{1/2-a_i-\alpha_i \gamma}\prod_{j=m+1}^q \beta_j^{b_j+\beta_j \gamma-1/2}\\
				& \cdot  \mathrm{e}^{-R \cos(\theta)\mathrm{Log}(|R|)a_2^* } \\
				& \cdot  \mathrm{e}^{R\cos(\theta)\mathrm{Log}\left(\prod_{i=1}^n\alpha_i^{-\alpha_i}\prod_{j=m+1}^q\beta_j^{\beta_j}\right)} \\
				& \cdot \mathrm{e}^{R(\cos(\theta)+\sin(\theta)\mathrm{sign}(\theta)(|\theta|-\pi))a_2^*} \\
				& \cdot  \mathrm{e}^{\left(\sum_{j=m+1}^q b_j-\sum_{i=1}^n a_i - \gamma a_2^* +(n+m-q)/2\right)\mathrm{Log}(|R|)}, \\
			\end{split}
		\end{equation*}
where $a_2^*=\sum_{i=1}^n \alpha_i - \sum_{j=m+1}^{q}\beta_j$, see Equation~(1.1.12) in \cite{SaiKil}.\\
Hence, we have that the following asymptotic formula holds for $\mathcal{H}^{m,n}_{p,q}(s)$, as $R \to \infty$, for $s=\gamma + R\mathrm{e}^{\mathrm{i}\theta}$, for $\gamma \in \mathbb{R}$ and $\theta \in (- \pi,\pi)\backslash \{0\}$, 
\begin{equation} \label{eq:CaligraphicHasymptotic}
			\begin{split}
				\mathcal{H}^{m,n}_{p,q}(s)& \sim  (2\pi)^{m+n-(p+q)/2}\prod_{j=1}^q \beta_j^{b_j+\beta_j\gamma-1/2} \prod_{i=1}^p \alpha_i^{1/2-a_i-\alpha_i\gamma} \mathrm{e}^{\Delta\cos(\theta)\mathrm{Log}(|R|)R}\\
				& \cdot \mathrm{e}^{\left[\cos(\theta)\mathrm{Log}\left(\delta\right)-\cos(\theta)\Delta-\sin(\theta)\left(a_1^*\theta -a_2^*\mathrm{sign}(\theta)(|\theta|-\pi)\right)\right]R}
				\\
				& \cdot \mathrm{e}^{\left( \gamma\Delta  +\mu \right)\mathrm{Log}(|R|)},
			\end{split}
		\end{equation}
		where $\delta=\prod_{j=1}^q\beta_j^{\beta_j} \prod_{i=1}^p\alpha_i^{-\alpha_i}$, $\mu=\sum_{j=1}^q b_j-\sum_{i=1}^p a_i +(p-q)/2$ and $\Delta=a_1^*-a_2^*=\sum_{j=1}^q\beta_j-\sum_{i=1}^p\alpha_i$ are given in Equations~(1.1.9)-(1.1.10)-(1.1.13) in \cite{SaiKil}, respectively.
		
		\item For the factor $|z^{-s}|$, with $z\neq0$ and $s=\gamma +R\mathrm{e}^{\mathrm{i}\theta}$, with $R>0$ and $\theta \in (-\pi,\pi]$, we have that 
\begin{equation}\label{eq:powersasymptotic}
			|z^{-s}| = |z|^{-\gamma}\mathrm{e}^{-R(\cos(\theta)\mathrm{Log}(|z|)+\sin(\theta)\arg(z))} \end{equation}
   by using the complex logarithm $\log(z)=\mathrm{Log}(|z|) + \mathrm{i}\arg(z)$, where $\arg(z)\in (-\pi,\pi]$.\\
		
		\item Finally, putting together \eqref{eq:CaligraphicHasymptotic} and \eqref{eq:powersasymptotic}, with $s=\gamma + R\mathrm{e}^{\mathrm{i}\theta}$ and $\theta \in (-\pi,\pi)\backslash\{0\}$, yields
\begin{equation*}
			\begin{split}
				|\mathcal{H}^{m,n}_{p,q}(s)z^{-s}|&\sim  C \mathrm{e}^{\Delta\cos(\theta)\mathrm{Log}(|R|)R}\\
				& \cdot \mathrm{e}^{R\cos(\theta)\mathrm{Log}\left(\delta\right)}\mathrm{e}^{-R\cos(\theta)\Delta}\mathrm{e}^{-R\cos(\theta)\mathrm{Log}(|z|)}\\
				& \cdot \mathrm{e}^{-R\sin(\theta)\left(a_1^*\theta -a_2^*\mathrm{sign}(\theta)(|\theta|-\pi)\right)}\mathrm{e}^{-R\sin(\theta)\arg(z)} \\
				& \cdot \mathrm{e}^{\left(\gamma\Delta+\mu \right)\mathrm{Log}(|R|)},
			\end{split}
		\end{equation*}
		as $R \to \infty$, where
  \[ C=|z|^{-\gamma}(2\pi)^{m+n-(p+q)/2} \prod_{j=1}^q \beta_j^{b_j+\beta_j\gamma-1/2} \prod_{i=1}^p \alpha_i^{1/2-a_i-\alpha_i\gamma}  \]
and the claim \eqref{eq:HKernelonCircle} follows, by defining the parameters $\Upsilon_1=\cos(\theta)\Delta$, 
\[\Upsilon_2=\cos(\theta)\mathrm{Log}(\exp(\Delta)|z|/\delta)+\sin(\theta)\Theta
\] 
and $\Upsilon_3=\mu+\gamma\Delta$, with $\Theta=a_1^* \theta - a_2^* \mathrm{sign}(\theta)(|\theta|-\pi)+\arg(z)$.\hfill \qedhere 
	\end{enumerate}
\end{proof}

\begin{remark} 
The result in Lemma~\ref{lem:asympHCircle} coincides with the asymptotic formula in Equations~(1.2.11) and (1.2.12) in \cite{SaiKil}, if $s$ is restricted to the contour $\mathcal{L}=l_j$, $j=1,2$, or $\mathcal{L}=l_3$, respectively. 

In Sec.~2.4 of \cite{Par01}, the leading term for $\mathrm{Log}(|\mathcal{H}^{m,n}_{p,q}(s)z^s|)$ is computed for $s=R \mathrm{e}^{\mathrm{i}\theta}$, with $\theta=\pm \pi/2$, and for $\theta \to 0$.
	
\end{remark}

\subsection{Analytic continuation}

The following theorem assumes the conditions under which the Fox-$H$ functions are well-defined, that is, their integral representation given in Definition \ref{def:H-function} converges. This can be derived from the asymptotic formula given in Lemma~\ref{lem:asympHCircle} and by Theorem 1.1 in \cite{SaiKil}.

In the following, we prove that Fox-$H$ functions $H^{m,n}_{p,q}(\cdot)$ with $a^*>0$, $\Delta>0$ defined via $\mathcal{L}=\mathcal{L}_{\mathrm{i}\gamma \infty}$ has an analytic continuation on $\mathbb{C}\backslash\{ 0\}$.

\begin{proof}[Proof of Theorem \ref{thm:ResidueDeltaPosi}]
	Let $m,n,p,q \in \mathbb{N}$ and $a_i, b_j \in \mathbb{R}$ and for $\alpha_i,\beta_j \in (0,\infty)$ with $i=1,\dots,p $ and $j=1,\dots,q$ and $\mathcal{L}=\mathcal{L}_{\mathrm{i} \gamma \infty}$ with $\gamma \in \mathbb{R}$ and the contour separates the poles $b_{jl}$ in Equations~\eqref{def:polesa} with respect to the poles $a_{ik}$ in Equation~\eqref{def:polesb} so that $b_{jl}$ are on its left and $a_{ik}$ on its right. \\
	We consider a positive and increasing sequence $\{R_L : L\in \mathbb{N}\}$ such that $R_L \to \infty$ as $L \to \infty$ and the following arc $C_L^{-}$ does not cross any of the poles $b_{jl}$
	\[  C_L^{-}:= \left\{ z \in \mathbb{C}\,\bigg|\, z=R_L \mathrm{e}^{\mathrm{i} \theta}+\gamma\, , \theta \in \left[ - \pi, -\frac{\pi}{2}\right] \cup \left[\frac{\pi}{2} , \pi\right)  \right\},\]
	and 
	\[ I_L:= \left\{ z \in \mathbb{C}\,\bigg|\, \Re(z)=\gamma \,, \Im(z) \in \left[-R_L,R_L \right] \right\}.\]
	We define the closed curves $\gamma_L:=I_L \cup C_L^{-}$ traversed counterclockwise and for each curve $\gamma_L$ we define a region $A_L$ that encloses the curve $\gamma_L$ and such that the poles of $\mathcal{H}^{m,n}_{p,q}(\cdot)$ are inside each $A_L$ coincide with those inside $\gamma_L$.
	We note that the above curves and regions always exist because each pole $b_{jl}$ is an isolated singularity; see Definition 3.3.2 in \cite{Mars87}. Moreover, each curve $\gamma_L$ contains a finite number of poles of $\mathcal{H}^{m,n}_{p,q}(\cdot)$ and, for each pole $b_{jl}$, the index number $I(\gamma_L,b_{jl})=1$ for any $L$.\\
	By the residue theorem, we have that for any $L \in \mathbb{N}$ and for any $z \in \mathbb{C}\backslash \{0\}$ such that $|\arg(z)|< a^* \pi/2$:
	
	\[ \int_{\gamma_L} \mathcal{H}^{m,n}_{p,q}(z) z^{-s} \mathrm{d}s= \sum_{b_{jl}
		\in A_L} \mathrm{Res}(\mathcal{H}^{m,n}_{p,q}(s)z^{-s},b_{jl}),  \]
	where the right side is given by Equation~(1.3.5) in \cite{SaiKil}. Furthermore, we have 
	\[ \int_{\gamma_L} \mathcal{H}^{m,n}_{p,q}(z) z^{-s} \mathrm{d}s=  \int_{I_L} \mathcal{H}^{m,n}_{p,q}(z) z^{-s} \mathrm{d}s +  \int_{C_L^{-}} \mathcal{H}^{m,n}_{p,q}(z) z^{-s} \mathrm{d}s\]
	and we have $\int_{C_L^{-}} |\mathcal{H}^{m,n}_{p,q}(s) z^{-s}| \mathrm{d}s \to 0$ as $L \to \infty$, in analogy with Section 14.5 in \cite{Whit21} and using the asymptotic formula in Lemma~\ref{lem:asympHCircle}. We note that if $\Delta>0$, then the leading term of $|\mathcal{H}^{m,n}_{p,q}(s)z^{-s}|$ is $\exp(\Delta \cos(\theta) R_L \mathrm{Log}(|R_L|) )$, with $\Delta \cos(\theta)<0$ for $\theta \in \left[ - \pi, -\frac{\pi}{2}\right) \cup \left(\frac{\pi}{2} , \pi\right)$.  \\
	Thus, for each $z \in \mathbb{C}$ such that $|\arg(z)|< a^* \pi/2$ we have
	\[  \C \ni H^{m,n}_{p,q}(z)=\lim_{L \to \infty}\int_{\gamma_L} \mathcal{H}^{m,n}_{p,q}(z) z^{-s} \mathrm{d}s= \sum_{b_{jl} }\mathrm{Res}(\mathcal{H}^{m,n}_{p,q}(s)z^{-s},b_{jl}).  \]
	The above series representation holds, for $z \in \mathbb{C}\backslash \{0 \}$ such that $|\arg(z)|< a^* \pi/2$, and coincides with the series representation of Theorem~1.2 i) in \cite{SaiKil}, in case of $z \in \mathbb{C}\backslash\{0\}$, $\Delta>0$, and $\mathcal{L}=\mathcal{L}_{-\infty}$.
	
	By applying Corollary~6.1.3 in \cite{Mars87}, the claim of the theorem is proved.
\end{proof}

\section{Fox-$H$ random variables}\label{sec:AppendixFoxHRandomVariable}
In this Appendix, we collect some known properties of Fox-$H$ random variables (see Definition \ref{defHdengeneral}) that are used in Section \ref{sec:FoxfunctionsAsDensities}. More precisely, the product and powers of these random variables are studied; for a further insight, see \cite{Spri80}.
\begin{theorem}[cf.~Thm.~4.1 in \cite{CarSprin}]\label{thmproductH}
	Let $X_1,X_2,\dots,X_N$ be independent Fox-$H$ random variables on the same probability space $(\Omega, \mathcal{F}, \mathbb{P})$ with probability density functions $\varrho_{X_1}(x_1),\dots,\varrho_{X_N}(x_N)$, respectively of the form 
	\[\varrho_{X_j}(x_j)=\frac{1}{K_j} H^{m_j,n_j}_{p_j,q_j}\left[ c_j x_j \, \Bigg|\genfrac{}{}{0pt}{}{(a_{j1},\alpha_{j1}),\dots,(a_{jp_j},\alpha_{jp_j})}{(b_{j1},\beta_{j1}),\dots,(b_{jq_j},\beta_{jq_j})}\right], \quad  x_j >0,  \]
	and zero otherwise, for $j=1,2,\dots,N$.\\
	Then the probability density function of the random variable
	\[  Y=\prod_{j=1}^N X_j  \]
	is given by
	\[ \varrho_Y(y)=\left( \prod_{j=1}^N \frac{1}{K_j} \right) H^{m',n'}_{p',q'}\left[ \left(\prod_{j=1}^N c_j\right) y  \, \Bigg|\genfrac{}{}{0pt}{}{(a_{11},\alpha_{11}),\dots,(a_{Np_N},\alpha_{Np_N})}{(b_{11},\beta_{11}),\dots,(b_{Nq_N},\beta_{Nq_N})}\right], \quad y>0, \]
	and zero otherwise, with 
	\[
	m'=\sum_{j=1}^N m_j,\quad  n'=\sum_{j=1}^N n_j,\quad p'=\sum_{j=1}^N p_j\quad  \mathrm{and}\quad q'=\sum_{j=1}^N q_j.
	\]
	The sequence of parameters is $(a_{jv},\alpha_{jv})$, $v=1,2,\dots,n_j$ for $j=1,2,\dots,N$, followed by $v=n_j+1,n_j+2,\dots, p_j$ for $j=1,2,\dots,N$, and the sequence of parameters $(b_{jv},\beta_{jv})$, $v=1,2,\dots,m_j$ for $j=1,2,\dots,N$, followed by $v=m_j+1,m_j+2,\dots,q_j$ for $j=1,2,\dots,N$.
\end{theorem}

We now re-state the following result using an alternative method.  

\begin{theorem}[cf.~Thm.~4.2 in \cite{CarSprin}]\label{thmpowertH0inf}
	Let $X$ be a Fox-$H$ random variable on the probability space $(\Omega, \mathcal{F}, \mathbb{P})$ with probability density  
	\begin{equation}
		\label{eq:Hdenrv}\varrho_X(x)=\frac{1}{K} H^{m,n}_{p,q}\left[ c x \, \Bigg|\genfrac{}{}{0pt}{}{(a_1,\alpha_1),\dots,(a_p,\alpha_p)}{(b_1,\beta_1),\dots,(b_q,\beta_q)}\right], \quad x >0  \end{equation}
	and zero otherwise.\\
	Then the probability density function of the random variable
	\[  Y:= X^P,\quad P\in \mathbb{R}\backslash \{0 \}\;  \]
	is given:
	\begin{itemize}
		\item[1)] for $P>0$, by
		\[ \varrho_Y(y)= \frac{c^{P-1}}{K}  H^{m,n}_{p,q}\left[ c^P y \, \Bigg|\genfrac{}{}{0pt}{}{(a_i-\alpha_iP+\alpha_i,\alpha_iP)_{1,p}}{(b_j-\beta_jP+\beta_j,\beta_jP)_{1,q}}\right], \]
		for $y>0$, and zero otherwise, 
		\item[2)] for $P<0$, by
		\[ \varrho_Y(y)= \frac{c^{P-1}}{K}  H^{n,m}_{q,p}\left[ c^{P} y  \, \Bigg|\genfrac{}{}{0pt}{}{(1-b_j+\beta_jP-\beta_j,-\beta_jP)_{1,q}}{(1-a_i+\alpha_iP-\alpha_i,\alpha_{i}P)_{1,p}}\right], \]
		for $y>0$, and zero otherwise.
	\end{itemize}

\end{theorem}

\begin{proof}
 For rational and positive $P$ it is given in Thm.~4.2 in \cite{CarSprin}. 
 For irrational and positive $P$ the proof can be made considering that the density of $Y:=X^P$ is given by $\varrho_{Y}(y)=y^{-1-1/P} \varrho_X(y^{1/P})/P$, $y>0$. The result follows by Equations~(2.1.4) and (2.1.5) in \cite{SaiKil}.
\end{proof}

\begin{remark}\label{rem:thmProductPowerHdensityon01}
\begin{itemize}
    \item[i)] Theorems \ref{thmproductH} and \ref{thmpowertH0inf} hold also for the Fox-$H$ random variable whose support is $(0,1)$. The proof follows easily by choosing the parameters $(a_{j1},\alpha_{j1}),\dots, (b_{jq_j},\beta_{jq_j})$ such that the corresponding Fox-$H$ function $H^{m_j,n_j}_{p_j,q_j}(x)$ is zero for $x>1$ for $j=1,\dots,N$, using the same argument as in \cite{Kir}, p.~13.
    \item[ii)] As we are interested in Fox-$H$ density with moments of any order, then in the above theorems we have $n,n_j=0$, $j=1,\dots,N$.
    \end{itemize}
\end{remark}

\subsection*{Acknowledgments}

The authors are grateful to Prof.~Francesco Mainardi for providing us with helpful material and for engaging in discussionsduring the "Workshop on Fractional Calculus, Special Functions, and Applications" and to G.~Pagnini and R.~Garra for the fruitful suggestions they give us. Furthermore, the research was partially carried over during a pleasant stay at the Isaac Newton Institute in Cambridge for the program Fractional Differential Equations  and also during the workshop "Linnaeus Workshop on Stochastic Analysis and Applications 2024". L.C. thanks Prof. Da Silva for his really kind hospitality within his research stay at the Centro de Ci\^encias Matem\'aticas in 2023.

\subsection*{Founding}
This work has been partially funded by the Center for Research in Mathematics
and Applications (CIMA) related to Statistics, Stochastic Processes
and Applications (SSPA) group, through the grant UIDB/MAT/04674/2020
of FCT-Funda{\c c\~a}o para a Ci{\^e}ncia e a Tecnologia, Portugal. The first and second authors were supported by the Prin Project 2022, n.202277N5H9.

\subsection*{Conflict of Interest Statement}

The authors declare that they have no known competing financial interests or personal relationships that could have appeared to influence the work reported in this paper.

\subsection*{Data Availability Statement}

Data sharing is not applicable to this article as no datasets were generated or analyzed during the current study.

\bibliography{FoxHDensityCompletelyMonotoneGeneralizedWrightFunctions}

\end{document}